\documentclass[3p,10pt]{elsarticle}
\usepackage{amssymb,latexsym,amsmath,color,subfigure,amsthm,url,charter}
\journal{}

\newcommand{\eps}{\varepsilon}

\newcommand{\set}[1]{\left\{#1\right\}}

\newcommand{\p}{\partial}
\newcommand{\mc}{\mathbf{c}}
\newcommand{\mn}{\mathbf{n}}

\newcommand{\mt}{\mathbf{t}}
\newcommand{\mx}{\mathbf{x}}
\newcommand{\my}{\mathbf{y}}
\newcommand{\mz}{\mathbf{z}}
\newcommand{\mH}{\mathbf{H}}
\newcommand{\mT}{\mathbf{T}}
\newcommand{\mU}{\mathbf{U}}
\newcommand{\mV}{\mathbf{V}}
\newcommand{\mW}{\mathbf{W}}

\newcommand{\W}{\mathbb{W}}

\newcommand{\vt}{\boldsymbol{\theta}}
\newcommand{\vv}{\boldsymbol{\vartheta}}
\newcommand{\vx}{\boldsymbol{\xi}}

\newtheorem{thm}{Theorem}[section]
\newtheorem{cor}[thm]{Corollary}
\newtheorem{lem}[thm]{Lemma}

\newtheorem{rem}[thm]{Remark}

\begin{document}

\begin{frontmatter}

%% Title, authors and addresses

%% use the tnoteref command within \title for footnotes;
%% use the tnotetext command for theassociated footnote;
%% use the fnref command within \author or \address for footnotes;
%% use the fntext command for theassociated footnote;
%% use the corref command within \author for corresponding author footnotes;
%% use the cortext command for theassociated footnote;
%% use the ead command for the email address,
%% and the form \ead[url] for the home page:
%% \title{Title\tnoteref{label1}}
%% \tnotetext[label1]{}
%% \author{Name\corref{cor1}\fnref{label2}}
%% \ead{email address}
%% \ead[url]{home page}
%% \fntext[label2]{}
%% \cortext[cor1]{}
%% \address{Address\fnref{label3}}
%% \fntext[label3]{}

\title{Subspace migration for imaging of thin, curve-like electromagnetic inhomogeneities without shape information}

%% use optional labels to link authors explicitly to addresses:
%% \author[label1,label2]{}
%% \address[label1]{}
%% \address[label2]{}

\author{Won-Kwang Park}
\ead{parkwk@kookmin.ac.kr}
\address{Department of Mathematics, Kookmin University, Seoul, 136-702, Korea.}

\begin{abstract}
It is well-known that subspace migration is stable and effective non-iterative imaging technique in inverse scattering problem. But, for a proper application, geometric features of unknown targets must be considered beforehand. Without this consideration, one cannot retrieve good results via subspace migration. In this paper, we identify the mathematical structure of single- and multi-frequency subspace migration without any geometric consideration of unknown targets and explore its certain properties. This is based on the fact that elements of so-called Multi-Static Response (MSR) matrix can be represented as an asymptotic expansion formula. Furthermore, based on the examined structure, we improve subspace migration and consider the multi-frequency subspace migration. Various results of numerical simulation with noisy data support our investigation.
\end{abstract}

\begin{keyword}
%% 5~6 different keywords......
Subspace migration \sep thin electromagnetic inhomogeneities \sep geometric consideration \sep Multi-Static Response (MSR) matrix \sep numerical simulation

%% keywords here, in the form: keyword \sep keyword

%% PACS codes here, in the form: \PACS code \sep code

%% MSC codes here, in the form: \MSC code \sep code
%% or \MSC[2008] code \sep code (2000 is the default)
\end{keyword}

\end{frontmatter}

%% \linenumbers

%% main text

%% The Appendices part is started with the command \appendix;
%% appendix sections are then done as normal sections
%% \appendix

%% \section{}
%% \label{}

\section{Preliminaries}\label{Sec1}
The purpose of this paper is to analyze subspace migration for imaging of thin crack-like electromagnetic inhomogeneity located in the two-dimensional homogeneous space $\mathbb{R}^2$. In order to properly start the analysis, let us introduce the mathematical model and subspace migration imaging algorithm before a brief recapitulation of known results and the presentation of the structure of the paper.

%\subsection{Direct scattering problem, far-field pattern, and asymptotic expansion formulae}
Let $\Gamma$ be a thin, curve-like homogeneous imhomogeneity within a homogeneous space $\mathbb{R}^2$. Throughout this paper, we assume that $\Gamma$ is localized in the neighborhood of a finitely long, smooth curve $\sigma$ such that
\begin{equation}\label{TI}
\Gamma=\set{\mx+\eta\mn(\mx):\mx\in\sigma,-h\leq\eta\leq h},
\end{equation}
where $\mn(\mx)$ is the unit normal to $\sigma$ at $\mx$, and $h$ is a strictly positive constant which specifies the thickness of the inhomogeneity (small with respect to the wavelength), refer to Figure \ref{ThinInclusion}. Throughout this paper, we denote $\mt(\mx)$ be the unit tangent vector at $\mx\in\sigma$.

\begin{figure}[h]
\begin{center}
\includegraphics[width=0.3\textwidth,keepaspectratio=true,angle=0]{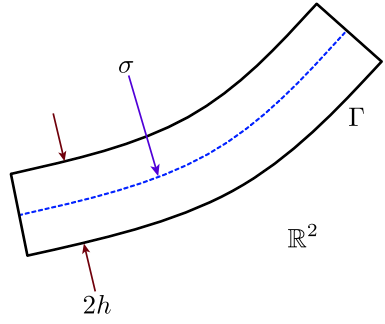}
\caption{\label{ThinInclusion}Sketch of the thin inhomogeneity $\Gamma$ in two-dimensional space $\mathbb{R}^2$.}
\end{center}
\end{figure}

In this paper, we assume that every material is characterized by its dielectric permittivity and magnetic permeability at a given
frequency. Let $0<\eps_0<+\infty$ and $0<\mu_0<+\infty$ denote the permittivity and permeability of the embedding space $\mathbb{R}^2$, and
$0<\eps_\star<+\infty$ and $0<\mu_\star<+\infty$ the ones of the inhomogeneity $\Gamma$. Then, we can define the following piecewise constant dielectric
permittivity
\begin{equation}\label{EP}
\eps(\mx)=\left\{\begin{array}{ccl}
\eps_0&\mbox{for}&\mx\in\mathbb{R}^2\backslash\overline{\Gamma}\\
\eps_\star&\mbox{for}&\mx\in\Gamma
\end{array}\right.
\end{equation}
and magnetic permeability
\begin{equation}\label{MP}
\mu(\mx)=\left\{\begin{array}{ccl}
\mu_0&\mbox{for}&\mx\in\mathbb{R}^2\backslash\overline{\Gamma}\\
\mu_\star&\mbox{for}&\mx\in\Gamma,
\end{array}\right.
\end{equation}
respectively. Note that if there is no inhomogeneity, i.e., in the homogeneous space, $\mu(\mx)$ and $\eps(\mx)$ are equal to $\mu_0$ and
$\eps_0$ respectively. In this paper, we set $\eps_\star>\eps_0=1$ and $\mu_\star>\mu_0=1$ for convenience but exact values of $\eps_\star$ and $\mu_\star$ are assumed unknown.

At strictly positive angular frequency $\omega$ (wavenumber $k_0=\omega\sqrt{\eps_0\mu_0}=\omega$), let $u_{\mathrm{tot}}(\mx;\omega)$ be the time-harmonic total field which satisfies the Helmholtz equation
\begin{equation}\label{HE1}
\nabla\cdot\left(\frac{1}{\mu(\mx)}\nabla u_{\mathrm{ tot}}(\mx;\omega)\right)+\omega^2\eps(\mx)u_{\mathrm{
tot}}(\mx;\omega)=0\quad\mbox{in}\quad\mathbb{R}^2
\end{equation}
with transmission condition on the boundary $\p\Gamma$. Similarly, the incident field $u_{\mathrm{inc}}(\mx;\omega)$ satisfies the homogeneous Helmholtz equation
\[\triangle u_{\mathrm{inc}}(\mx;\omega)+\omega^2u_{\mathrm{inc}}(\mx;\omega)=0\quad\mbox{in}\quad\mathbb{R}^2.\]
Throughout this paper, we consider the illumination of plane waves
\[u_{\mathrm{inc}}(\mx;\omega)=e^{i\omega\vt\cdot\mx}\quad\mbox{for}\quad\mx\in\mathbb{R}^2,\]
where $\vt$ is a two-dimensional vector, which characterizes the direction, on the unit circle $\mathbb{S}^1$ in $\mathbb{R}^2$.
As is usual, the total field $u_{\mathrm{tot}}(\mx;\omega)$ divides itself into the incident field $u_{\mathrm{inc}}(\mx;\omega)$ and the corresponding scattered field $u_{\mathrm{scat}}(\mx;\omega)$ such that $u_{\mathrm{ tot}}(\mx;\omega)=u_{\mathrm{inc}}(\mx;\omega)+u_{\mathrm{scat}}(\mx;\omega)$. Notice that this unknown
scattered field $u_{\mathrm{scat}}(\mx;\omega)$ satisfies the Sommerfeld radiation condition
\[\lim_{|\mx|\to\infty}\sqrt{|\mx|}\left(\frac{\p u_{\mathrm{ scat}}(\mx;\omega)}{\p|\mx|}-i\omega
u_{\mathrm{scat}}(\mx;\omega)\right)=0\]
uniformly in all directions $\hat{\mx}=\mx/|\mx|$.

The far-field pattern is defined as a function $u_{\infty}(\vv,\vt;\omega)$ which satisfies \[u_{\infty}(\vv,\vt;\omega)=\frac{e^{i\omega|\my|}}{\sqrt{|\my|}}u_{\mathrm{scat}}(\mx;\omega)+o\left(\frac{1}{\sqrt{|\my|}}\right)\]
as $|\my|\longrightarrow\infty$ uniformly on $\vv=\my/|\my|\in\mathbb{S}^1$ and $\vt\in\mathbb{S}^1$. Then, based on {BF},
$u_{\infty}(\vv,\vt;\omega)$ can be written as an asymptotic expansion formula.

\begin{lem}[See \cite{BF}]
  For $\vv,\vt\in\mathbb{S}^1$ and $\mx\in\mathbb{R}^2\backslash\overline{\Gamma}$, the far-field pattern $u_{\infty}(\mx,\vt;\omega)$
  can be represented as
  \[u_{\infty}(\vv,\vt;\omega)=h\frac{\omega^2(1+i)}{4\sqrt{\omega\pi}}\int_{\sigma}\bigg((\eps_\star-1)-2\vv\cdot\mathbb{M}(\my)\cdot\vt\bigg)e^{i\omega(\vt-\vv)\cdot\my}d\sigma(\my),\]
  where $o(h)$ is uniform in $\my\in\sigma$, $\vv,\vt\in\mathbb{S}^1$, and $\mathbb{M}(\my)$ is a $2\times2$ symmetric matrix defined as
  follows: let $\mt(\my)$ and $\mn(\my)$ denote unit tangent and normal vectors to $\sigma$ at $\my$, respectively. Then
    \begin{itemize}
      \item $\mathbb{M}(\my)$ has eigenvectors $\mt(\my)$ and $\mn(\my)$.
      \item The eigenvalue corresponding to $\mt(\my)$ is $2\left(\frac{1}{\mu_\star}-\frac{1}{\mu_0}\right)=2\left(\frac{1}{\mu_\star}-1\right)$.
      \item The eigenvalue corresponding to $\mn(\my)$ is $2\left(\frac{1}{\mu_0}-\frac{\mu_\star}{\mu_0^2}\right)=2(1-\mu_\star)$.
    \end{itemize}
\end{lem}

Now, we introduce subspace migration for imaging of thin inhomogeneity $\Gamma$. Detailed description can be found in \cite{AGKPS,P-SUB3}. Let $\mathbb{K}(\omega)\in\mathbb{C}^{N\times N}$ be the Multi-Static Response (MSR) matrix whose elements are the collected far-field at observation number $j$ for the incident number $l$ such that
\[\mathbb{K}(\omega)=[K_{jl}(\omega)]_{j,l=1}^{N}=\left[
   \begin{array}{cccc}
     u_\infty(\vv_1,\vt_1) & u_\infty(\vv_1,\vt_2) & \cdots & u_\infty(\vv_1,\vt_N) \\
     u_\infty(\vv_2,\vt_1) & u_\infty(\vv_2,\vt_2) & \cdots & u_\infty(\vv_2,\vt_N) \\
     \vdots & \vdots & \ddots & \vdots \\
     u_\infty(\vv_N,\vt_1) & u_\infty(\vv_N,\vt_2) & \cdots & u_\infty(\vv_N,\vt_N)
   \end{array}
 \right].\]
In this paper, we assume that $\vv_j=-\vt_j$, i.e., we have the same incident and observation directions configuration. It is worth emphasizing that for a given frequency $\omega=2\pi/\lambda$, based on the resolution limit, any detail less than one-half of the wavelength cannot be retrieved. Hence, if we divide thin inhomogeneity $\Gamma$ into $M$ different segments of size of order $\lambda/2$, only one point, say, $\mx_m$, $m=1,2,\cdots,M$, at each segment will affect the imaging (see \cite{ABC,PL3}).
If $3M<N$, the elements of MSR matrix can be represented as follows:
\begin{align}
\begin{aligned}\label{ElementofMSR}
  u_{\infty}(\vv_j,\vt_l;\omega)=h\frac{\omega^2(1+i)}{4\sqrt{\omega\pi}}&\int_{\sigma}\bigg((\eps_\star-1)-2\left(\vv\cdot\mathbb{M}(\my)\cdot\vt\right)\bigg)e^{i\omega(\vt-\vv)\cdot\my}d\sigma(\my)\\
  \approx h\frac{\omega^2(1+i)}{4\sqrt{\omega\pi}}&\frac{|\sigma|}{M}\sum_{m=1}^{M}\bigg[(\eps_\star-1)+\left(\frac{1}{\mu_\star}-1\right)\vt_j\cdot\mt(\my_m)\vt_l\cdot\mt(\my_m)\\
  &+\left(1-\mu_\star\right)\vt_j\cdot\mn(\my_m)\vt_l\cdot\mn(\my_m)\bigg]e^{i\omega(\vt_j+\vt_l)\cdot\my_m},
\end{aligned}
\end{align}
where $|\sigma|$ denotes the length of $\sigma$.

Based on the representation (\ref{ElementofMSR}), $\mathbb{K}(\omega)$ can be decomposed as follows:
\begin{equation}\label{SymmetryMSR}
  \mathbb{K}(\omega)=\mathbb{H}(\omega)\mathbb{B}(\omega)\overline{\mathbb{H}}(\omega),
\end{equation}
where $\mathbb{B}(\omega)\in\mathbb{R}^{3M\times3M}$ is a block diagonal matrix with components
\[\mathbb{B}(\omega)=\frac{\omega^2|\sigma|}{M}
\left[
  \begin{array}{cc}
    \eps_\star-\eps_0 & \mathbb{O}_{1\times2} \\
    \mathbb{O}_{2\times1} & \mathbb{M}(\mx_m)\\
  \end{array}
\right],\]
and $\mathbb{H}(\omega)\in\mathbb{C}^{N\times3M}$ is written as
\[\mathbb{H}(\omega)=\bigg[\mathbb{H}_1(\omega),\mathbb{H}_2(\omega),\cdots,\mathbb{H}_M(\omega)\bigg].\]
Here, $\mathbb{O}_{p\times q}$ denotes the $p\times q$ zero matrix and vectors and $\mH_m(\omega)$ is represented
\begin{align}
\begin{aligned}\label{Vec1}
  \mathbb{H}_m(\omega)&=\bigg[\mH_m^{(1)}(\omega),\mH_m^{(2)}(\omega),\mH_m^{(3)}(\omega)\bigg]\\
  &=\left[
        \begin{array}{ccc}
          e^{i\omega\vt_1\cdot\mx_m}, & \vt_1\cdot\mt(\mx_m)e^{i\omega\vt_1\cdot\mx_m}, & \vt_1\cdot\mn(\mx_m)e^{i\omega\vt_1\cdot\mx_m} \\
          e^{i\omega\vt_2\cdot\mx_m}, & \vt_2\cdot\mt(\mx_m)e^{i\omega\vt_2\cdot\mx_m}, & \vt_2\cdot\mn(\mx_m)e^{i\omega\vt_2\cdot\mx_m} \\
          \vdots & \vdots & \vdots \\
          e^{i\omega\vt_N\cdot\mx_m}, & \vt_N\cdot\mt(\mx_m)e^{i\omega\vt_N\cdot\mx_m}, & \vt_N\cdot\mn(\mx_m)e^{i\omega\vt_N\cdot\mx_m} \\
        \end{array}
      \right].
\end{aligned}
\end{align}

Now, let us perform the Singular Value Decomposition (SVD) of $\mathbb{K}(\omega)$
\[\mathbb{K}(\omega)=\mathbb{U}(\omega)\mathbb{S}(\omega)\overline{\mathbb{V}}(\omega)^T=\sum_{m=1}^{N}\rho_m(\omega)\mU_m(\omega)\overline{\mV}_m(\omega)^T\approx\sum_{m=1}^{3M}\rho_m(\omega)\mU_m(\omega)\overline{\mV}_m(\omega)^T,\]
where $\rho_m(\omega)$, $m=1,2,\cdots,3M$, are nonzero singular values such that
\[\rho_1(\omega)\geq\rho_2(\omega)\geq\cdots\geq\rho_{3M}(\omega)>0\quad\mbox{and}\quad\rho_m(\omega)\approx0\quad\mbox{for}\quad
m\geq3M+1,\]
and $\mU_m(\omega)$ and $\mV_m(\omega)$ are left- and right-singular vectors of $\mathbb{K}(\omega)$, respectively. Based on the structure of (\ref{Vec1}), define a test vector $\mT(\mz;\omega)\in\mathbb{C}^{N\times1}$ as
\begin{equation}\label{VecT}
  \mT(\mz;\omega)=\bigg[\mc_1\cdot[1,\vt_1]e^{i\omega\vt_1\cdot\mz},\mc_2\cdot[1,\vt_2]e^{i\omega\vt_2\cdot\mz},\cdots,
  \mc_N\cdot[1,\vt_N]e^{i\omega\vt_N\cdot\mz}\bigg]^T
\end{equation}
and corresponding unit vector
\[\hat{\mT}(\mz;\omega):=\frac{\mT(\mz;\omega)}{|\mT(\mz;\omega)|},\]
where the selection of $\mc_n\in\mathbb{R}^3\backslash\set{\mathbf{0}}$, $n=1,2,\cdots,N$, is depending on the $\mt(\mx_m)$ and $\mn(\mx_m)$, i.e., shape of $\sigma$ (see
\cite{PL3} for a detailed discussion). Then, \textit{for a proper choice of} $\mc_n$, we can observe that
\[\hat{\mT}(\my_m;\omega)\approx\mU_m(\omega)\quad\mbox{and}\quad\hat{\mT}(\my_m;\omega)\approx\overline{\mV}_m(\omega).\]
Since, based on the orthonormal property of singular vectors, the first $3M$ columns of
$\mathbb{U}(\omega)$ and $\mathbb{V}(\omega)$ are orthonormal, it follows that
\begin{align}
\begin{aligned}\label{orthonormal}
  &\langle\hat{\mT}(\mz;\omega),\mU_m(\omega)\rangle\approx1,\quad\langle\hat{\mT}(\mz;\omega),\overline{\mV}_m(\omega)\rangle\approx1\quad\mbox{if}\quad\mz=\my_m\\
  &\langle\hat{\mT}(\mz;\omega),\mU_m(\omega)\rangle\approx0,\quad\langle\hat{\mT}(\mz;\omega),\overline{\mV}_m(\omega)\rangle\approx0\quad\mbox{if}\quad\mz\ne\my_m,
\end{aligned}
\end{align}
where $\langle\mathbf{a},\mathbf{b}\rangle=\overline{\mathbf{a}}\cdot\mathbf{b}$ for $\mathbf{a},\mathbf{b}\in\mathbb{C}$.

Hence, we can introduce subspace migration for imaging of thin inhomogeneity at a given frequency $\omega$ as
\begin{equation}\label{SingleSM}
  \W_{\mathrm{SM}}(\mz;\omega):=\left|\sum_{m=1}^{M}\langle\hat{\mT}(\mz;\omega),\mU_m(\omega)\rangle\langle\hat{\mT}(\mz;\omega),\overline{\mV}_m(\omega)\rangle\right|.
\end{equation}
Based on the properties (\ref{orthonormal}), map of $\W_{\mathrm{SF}}(\mz;\omega)$ should exhibit peaks of magnitude $1$ at $\mz=\mx_m\in\sigma$, and of small magnitude at $\mz\in\mathbb{R}^2\backslash\overline{\Gamma}$. This is the reason why thin inhomogeneity can be imaged via subspace migration.

Based on above, defining a vector $\mT(\mz;\omega)$ in (\ref{VecT}) plays a key role of imaging performance. For this, a proper selection of test vectors $\mc_n$ is very important. Note that based on (\ref{ElementofMSR}) and (\ref{Vec1}), $\mc_n$ must be a linear combination of tangential $\mt(\mx_m)$ and normal $\mn(\mx_m)$ vectors at $\mx_m\in\sigma$. However, we have no \textit{a priori} information of shape of thin inhomogeneity $\Gamma$, it is impossible to define an optimal vector $\mT(\mz;\omega)$. Due to this reason, in many works \cite{PL3,P-SUB1}, $\mc_n$ has chosen as a fixed vector and corresponding subspace migration imaging functional is considered.

In this paper, we consider the subspace migration imaging functional without any consideration of geometric property of thin electromagnetic inhomogeneities. Based on the structure of elements of MSR matrix (\ref{ElementofMSR}), we explore a relationship between subspace migration imaging functional and Bessel functions of integer order of the first kind. This relationship leads us certain properties of subspace migration and gives an idea of improvement of imaging performance. Furthermore, we extend such analysis to the multi-frequency subspace migration.

This paper is organized as follows. In Section \ref{Sec2}, we analyze single- and multi-frequency subspace migration imaging functional without any \textit{a priori} information of thin inhomogeneities by establishing a relationship with Bessel function of integer order of the first kind, discuss certain properties of subspace migration, and introduce an improved subspace migration. In Section \ref{Sec3}, several results of numerical experiments with noisy data are presented in order to support our analysis. Finally, a short conclusion is mentioned in Section \ref{Sec4}.

\section{Analysis of subspace migration without geometric consideration}\label{Sec2}
We now identify the structure of (\ref{SingleSM}) without consideration of shape of $\Gamma$. For this, since we have no information of $\mt(\my)$ and $\mn(\my)$ for $\my\in\sigma$, we consider the following test vector instead of (\ref{VecT})
\begin{equation}\label{TestVector}
  \mW(\mz;\omega)=\frac{1}{\sqrt{N}}\bigg[e^{i\omega\vt_1\cdot\mz},e^{i\omega\vt_2\cdot\mz},\cdots,e^{i\omega\vt_N\cdot\mz}\bigg]^T,
\end{equation}
and analyze corresponding single-frequency subspace migration functional
\begin{equation}\label{SubspaceMigration}
  \W_{\mathrm{SF}}(\mz;\omega):=\left|\sum_{m=1}^{M}\langle\mW(\mz;\omega),\mU_m(\omega)\rangle\langle\mW(\mz;\omega),\overline{\mV}_m(\omega)\rangle\right|.
\end{equation}
For starting analysis, we shall introduce two useful identities derived in \cite{P-SUB3}.

\begin{lem}\label{TheoremBessel}
  Let $\vx\in\mathbb{R}^2$ and $\vt_n\in\mathbb{S}^1$, $n=1,2,\cdots,N$. Then for sufficiently large $N$, the following relations hold:
  \begin{align*}
    &\frac{1}{N}\sum_{n=1}^{N}e^{i\omega\vt_n\cdot\mx}=\frac{1}{2\pi}\int_{\mathbb{S}^1}e^{i\omega\vt\cdot\mx}d\vt=J_0(\omega|\mx|)\\
    &\frac{1}{N}\sum_{n=1}^{N}\langle\vt_n,\vx\rangle e^{i\omega\vt_n\cdot\mx}=\frac{1}{2\pi}\int_{\mathbb{S}^1}\langle\vt,\vx\rangle e^{i\omega\vt\cdot\mx}d\vt=i\left\langle\frac{\mx}{|\mx|},\vx\right\rangle J_1(\omega|\mx|),
  \end{align*}
  where $J_n(\cdot)$ denotes the Bessel function of integer order $n$ of the first kind.
\end{lem}

\subsection{Analysis of single-frequency subspace migration}
Based on (\ref{TestVector}) and Lemma \ref{TheoremBessel}, we obtain the following main result.
\begin{thm}\label{TheoremSingleSM}
  For sufficiently large $N(>3M)$ and $\omega$, $\W_{\mathrm{SF}}(\mz;\omega)$ can be represented as follows:
  \begin{equation}\label{StructureSingleSM}
    \W_{\mathrm{SF}}(\mz;\omega)=\left|\sum_{m=1}^{M}\left\{J_0(\omega|\mx_m-\mz|)^2-2\left\langle\frac{\mx_m-\mz}{|\mx_m-\mz|},\mt(\mx_m)+\mn(\mx_m)\right\rangle^2J_1(\omega|\mx_m-\mz|)^2\right\}\right|.
  \end{equation}
  Furthermore, if $\omega\longrightarrow+\infty$ then
  \[\W_{\mathrm{SF}}(\mz;\omega)\approx\delta(\mx_m-\mz),\]
  where $\delta$ is the Dirac delta function.
\end{thm}
\begin{proof}
  Since $\mH_m^{(1)}$, $\mH_m^{(2)}$ and $\mH_m^{(3)}$ are orthogonal for $m=1,2,\cdots,M$ (see Section \ref{secA}), applying (\ref{TestVector}), and Lemma \ref{TheoremBessel} yields
  \[\W_{\mathrm{SF}}(\mz;\omega)=\left|\sum_{m=1}^{3M}\left\langle\mathbf{W}(\mz;\omega),\mU_m\right\rangle\left\langle\mathbf{W}(\mz;\omega),\overline{\mV}_m\right\rangle\right|\approx\left|\sum_{m=1}^{M}\sum_{s=1}^{3}\left\langle\mathbf{W}(\mz;\omega),\hat{\mH}_m^{(s)}(\omega)\right\rangle^2\right|,\]
  where
  \begin{align*}
    \hat{\mH}_m^{(1)}(\omega)&=\frac{1}{\sqrt{N}}\bigg[e^{i\omega\vt_1\cdot\mx_m},e^{i\omega\vt_2\cdot\mx_m},\cdots,e^{i\omega\vt_N\cdot\mx_m}\bigg]^T\\
    \hat{\mH}_m^{(2)}(\omega)&=\frac{\sqrt2}{\sqrt{N}}\bigg[\vt_1\cdot\mt(\mx_m)e^{i\omega\vt_1\cdot\mx_m},\vt_2\cdot\mt(\mx_m)e^{i\omega\vt_2\cdot\mx_m},\cdots,\vt_N\cdot\mt(\mx_m)e^{i\omega\vt_N\cdot\mx_m}\bigg]^T\\
    \hat{\mH}_m^{(3)}(\omega)&=\frac{\sqrt2}{\sqrt{N}}\bigg[\vt_1\cdot\mn(\mx_m)e^{i\omega\vt_1\cdot\mx_m},\vt_2\cdot\mn(\mx_m)e^{i\omega\vt_2\cdot\mx_m},\cdots,\vt_N\cdot\mn(\mx_m)e^{i\omega\vt_N\cdot\mx_m}\bigg]^T.
  \end{align*}
  Then, following an elementary calculus, we can calculate
  \begin{align}
  \begin{aligned}\label{TermH}
    \left\langle\mathbf{W}(\mz;\omega),\hat{\mH}_m^{(1)}(\omega)\right\rangle=&\frac{1}{N}\sum_{n=1}^{N}e^{i\omega\vt_n\cdot(\mx_m-\mz)}=J_0(\omega|\mx_m-\mz|)\\
    \left\langle\mathbf{W}(\mz;\omega),\hat{\mH}_m^{(2)}(\omega)\right\rangle=&\frac{\sqrt{2}}{N}\sum_{n=1}^{N}\langle\vt_n,\mt(\mx_m)\rangle e^{i\omega\vt_n\cdot(\mx_m-\mz)}=\sqrt{2}i\left\langle\frac{\mx_m-\mz}{|\mx_m-\mz|},\mt(\mx_m)\right\rangle J_1(\omega|\mx_m-\mz|)\\
    \left\langle\mathbf{W}(\mz;\omega),\hat{\mH}_m^{(3)}(\omega)\right\rangle=&\frac{\sqrt{2}}{N}\sum_{n=1}^{N}\langle\vt_n,\mn(\mx_m)\rangle e^{i\omega\vt_n\cdot(\mx_m-\mz)}=\sqrt{2}i\left\langle\frac{\mx_m-\mz}{|\mx_m-\mz|},\mn(\mx_m)\right\rangle J_1(\omega|\mx_m-\mz|)
  \end{aligned}
  \end{align}
  Hence, (\ref{StructureSingleSM}) can be derived by (\ref{TermH}).

  Now, assume that $\omega\longrightarrow+\infty$. Then, it is clear that $\W_{\mathrm{SF}}(\mz;\omega)\approx1$ when $\mz=\mx_m$. If $\mz\ne\mx_m$, then following asymptotic form of Bessel function holds for $\omega|\mx_m-\mz|\gg|n^2-0.25|$,
  \[J_n(\omega|\mx_m-\mz|)\approx\sqrt{\frac{2}{\omega\pi|\mx_m-\mz|}}\cos\left\{\omega|\mx_m-\mz|-\frac{n\pi}{2}-\frac{\pi}{4} +\mathcal{O}\left(\frac{1}{\omega|\mx_m-\mz|}\right)\right\}\longrightarrow0,\]
  where $n$ denotes a positive integer. Based on this asymptotic form, we can easily observe that $\W_{\mathrm{SF}}(\mz;\omega)\approx0$ when $\mz\ne\mx_m$. Hence,
  \[\W_{\mathrm{SF}}(\mz;\omega)\approx\delta(\mx_m-\mz).\]
\end{proof}

Based on the identified structure (\ref{StructureSingleSM}), we can examine certain properties of $\W_{\mathrm{SF}}(\mz;\omega)$:

\begin{enumerate}\renewcommand{\theenumi}{[P\arabic{enumi}]}
  \item\label{P1} Based on recent work \cite{PL3}, the dominant eigenvectors of matrix $\mathbb{M}(\mx)$ are $\mt(\mx)$ and $\mn(\mx)$ for $\mu_\star<\mu_0$ and $\mu_\star>\mu_0$, respectively. Generally (and based on our assumption), since $\mu_\star>\mu_0$, $\W_{\mathrm{SF}}(\mz;\omega)$ can be written as
      \[\W_{\mathrm{SF}}(\mz;\omega)\approx\left|\sum_{m=1}^{M}\left\{J_0(\omega|\mx_m-\mz|)^2-2\left\langle\frac{\mx_m-\mz}{|\mx_m-\mz|},\mn(\mx_m)\right\rangle^2J_1(\omega|\mx_m-\mz|)^2\right\}\right|.\]
  \item\label{P2} Since $J_0(0)=1$ and $J_n(0)=0$ for $n=1,2,\cdots,$ the terms
  \[J_0(\omega|\mx_m-\mz|)^2\quad\mbox{and}\quad2\left\langle\frac{\mx_m-\mz}{|\mx_m-\mz|},\mn(\mx_m)\right\rangle^2J_1(\omega|\mx_m-\mz|)^2\]
  contribute to and disturb the imaging performance, respectively.
  \item\label{P-SUB3} Based on the properties of $J_0(x)$ and $J_1(x)$, plots of $\W_{\mathrm{SF}}(\mz;\omega)$ will show peaks of magnitude $1$ at $\mz=\mx_m\in\Gamma$ and small one at $\mz\notin\Gamma$. Note that since $J_1(x)^2$ has its maximum value at two points, say $x_1$ and $x_2$, symmetric with respect to $x=0$, two curves with large (but less than $1$) magnitude and many artifacts with small magnitude will included in the map of $\W_{\mathrm{SF}}(\mz;\omega)$. This tells us that one can recognize the shape of thin inhomogeneities via the map of $\W_{\mathrm{SF}}(\mz;\omega)$ without geometric consideration.
\end{enumerate}

It is worth mentioning that based on \ref{P2}, $\W_{\mathrm{SF}}(\mz;\omega)$ has its maximum value at $\mz=\mx_m\in\Gamma$. Hence, we can immediately examine following result of unique determination.
\begin{cor}\label{CorollaryUnique1}
   Let the applied frequency $\omega$ be sufficiently high. If the total number $N$ of incident and observation directions is sufficiently large, then the shape of supporting curve $\sigma$ of thin inclusion $\Gamma$ can be obtained uniquely via the map of $\W_{\mathrm{SF}}(\mz;\omega)$.
\end{cor}

\subsection{Improvement of subspace migration part 1: filtering}
Now, we consider the method of improvement. Based on the structure (\ref{StructureSingleSM}), we can examine that good results can be obtained via the map of $\W_{\mathrm{SF}}(\mz;\omega)$ when $\omega\longrightarrow+\infty$. However, this is a theoretically ideal situation. Hence, for obtaining good results, an alternative method must be considered.

Based on the property discussed in \ref{P-SUB3}, one can obtain good results by eliminating two curves with large magnitude and correspondingly many artifacts with small magnitude. Hence, for designing a filtering strategy, let us consider the maximum value of the following term:
\[2\left\langle\frac{\mx_m-\mz}{|\mx_m-\mz|},\mn(\mx_m)\right\rangle^2J_1(\omega|\mx_m-\mz|)^2.\]
Since
\[\frac{d}{dx}J_1(x)^2=-2J_1(x)J_1'(x),\]
and the first zero of $J_1(x)$ and $J_1'(x)$ are approximately $x=3.8317$ and $x=1.8412$, respectively, $J_1(x)^2$ has its maximum at $x\approx1.8412$. Hence, based on the numerical computation,
\[\max_{\mz\in\mathbb{R}^2}J_1(\omega|\mx_m-\mz|)^2\approx(0.58186522\ldots)^2\approx0.338567,\]
and
\[2\left\langle\frac{\mx_m-\mz}{|\mx_m-\mz|},\mn(\mx_m)\right\rangle^2\leq2,\]
we can say that
\[2\left\langle\frac{\mx_m-\mz}{|\mx_m-\mz|},\mn(\mx_m)\right\rangle^2J_1(\omega|\mx_m-\mz|)^2\approx0.677134268\ldots<0.678.\]
This means that the magnitude of two curves will be less than $0.678$, refer to Figure \ref{Bessel12}. Hence, let us introduce a filtering function $\mathcal{F}_\mathrm{S}$ such that
\[\mathcal{F}_\mathrm{S}[x]=\left\{\begin{array}{ccc}
\medskip x & \mbox{if} & 0.678\leq x\leq1 \\
0 & \mbox{if} & 0\leq x<0.678.
\end{array}
\right.\]
Then, better imaging results of thin inhomogeneity can be obtained via the map of $\mathcal{F}_\mathrm{S}[\W_{\mathrm{SF}}(\mz;\omega)]$. Note that this method can be applied in the imaging of single inhomogeneity, refer to Figure \ref{GammaM1}.

\begin{figure}[h]
\begin{center}
\includegraphics[width=0.99\textwidth]{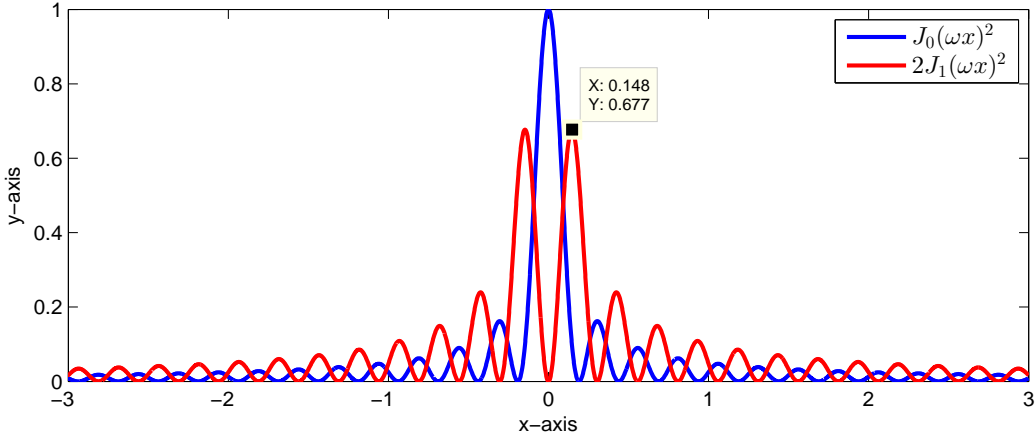}
\caption{\label{Bessel12}Graphs of $J_0(\omega x)^2$ and $2J_1(\omega x)^2$ for $\omega=2\pi/0.5$.}
\end{center}
\end{figure}

\begin{rem}\label{Remark1}
  Theoretically, the maximum value of $\W_{\mathrm{SF}}(\mz;\omega)$ is equal to $1$ but generally, maximum value is smaller than $1$ in the results of numerical simulations, refer to Figure \ref{GammaM1}. There are many reasons e.g., the algorithm is based on the asymptotic expansion formula, influence of random noise, and computational errors. So, instead of $\mathcal{F}_\mathrm{S}[\W_{\mathrm{SF}}(\mz;\omega)]$, we suggest the following normalized value
  \[\hat{\W}_{\mathrm{SF}}(\mz;\omega):=\frac{\W_{\mathrm{SF}}(\mz;\omega)}{\displaystyle\max_{\mz\in\mathbb{R}^2}|\W_{\mathrm{SF}}(\mz;\omega)|}\]
  and consider the filtered map $\mathcal{F}_\mathrm{S}[\hat{\W}_{\mathrm{SF}}(\mz;\omega)]$ for identifying shape of inclusions.
\end{rem}

\subsection{Improvement of subspace migration part 2: application of multi-frequency}
Based on recent works \cite{AGKPS,P-SUB3}, it has confirmed that applying multi-frequency offers better results than applying single frequency. But this fact holds with an optimal choice of $\mc_n$ in (\ref{VecT}). Now, we consider the multi-frequency subspace migration and examine that it improves single-frequency one. Let $\{\omega_f=2\pi/\lambda_f:f=1,2,\cdots,F\}$ be the set of $F-$different angular frequencies such that
\[\omega_1<\omega_2<\cdots<\omega_F,\quad\mbox{i.e.},\quad\lambda_1>\lambda_2>\cdots>\lambda_F,\]
and let $\mathbb{K}(\omega_f)$ be the collected MSR matrix at $\omega_f$. Then, by performing SVD of $\mathbb{K}(\omega_f)$ as
\[\mathbb{K}(\omega_f)=\mathbb{U}(\omega_f)\mathbb{S}(\omega_f)\overline{\mathbb{V}}(\omega_f)^T=\sum_{m=1}^{N}\rho_m(\omega_f)\mU_m(\omega_f)\overline{\mV}_m(\omega_f)^T\approx\sum_{m=1}^{3M_f}\rho_m(\omega_f)\mU_m(\omega_f)\overline{\mV}_m(\omega_f)^T,\]
we can introduce multi-frequency subspace migration:
\begin{equation}\label{MultiSM}
  \W_{\mathrm{MM}}(\mz;F):=\frac{1}{F}\left|\sum_{f=1}^{F}\W_{\mathrm{SM}}(\mz;\omega)\right|=\frac{1}{F}\left|\sum_{f=1}^{F}\sum_{m=1}^{M_f}\langle\hat{\mT}(\mz;\omega_f),\mU_m(\omega_f)\rangle\langle\hat{\mT}(\mz;\omega_f),\overline{\mV}_m(\omega_f)\rangle\right|,
\end{equation}
and corresponding multi-frequency subspace migration without geometric consideration:
\begin{equation}\label{SubspaceMigrationMulti}
  \W_{\mathrm{MF}}(\mz;F):=\frac{1}{F}\left|\sum_{f=1}^{F}\W_{\mathrm{SF}}(\mz;\omega_f)\right|=\frac{1}{F}\left|\sum_{f=1}^{F}\sum_{m=1}^{M_f}\langle\mW(\mz;\omega_f),\mU_m(\omega_f)\rangle\langle\mW(\mz;\omega_f),\overline{\mV}_m(\omega_f)\rangle\right|,
\end{equation}
where $\hat{\mT}(\mz;\omega_f)$ and $\mW(\mz;\omega_f)$ are defined in (\ref{VecT}) and (\ref{TestVector}), respectively.

From the results in several works \cite{AGKPS,P-SUB3,P-SUB1}, it has confirmed that (\ref{MultiSM}) is an improved imaging function of (\ref{SingleSM}). This is based on the Statistical Hypothesis Testing \cite{AGKPS}, several results of numerical experiments, and a relationship with Bessel functions \cite{P-SUB3,P-SUB1}. However, there is no theoretical results about the improvement of (\ref{SubspaceMigrationMulti}). In this section, we identify the reason by establishing a relationship with Bessel functions of integer order as follows.

\begin{thm}\label{TheoremMultiSM}
  For sufficiently large $N(>3M)$ and $\omega$, $\W_{\mathrm{MF}}(\mz;\omega)$ can be represented as follows:
  \begin{multline}\label{StructureMultiSM}
    \W_{\mathrm{MF}}(\mz;F)\approx\frac{1}{\omega_F-\omega_1}\left|\sum_{m=1}^{M}\bigg\{\Lambda(\mx_m-\mz;\omega_F)-\Lambda(\mx_m-\mz;\omega_1)\right.\\
    +\left.\left(1-2\left\langle\frac{\mx_m-\mz}{|\mx_m-\mz|},\mn(\mx_m)\right\rangle^2\right)\int_{\omega_1}^{\omega_F}J_1(\omega|\mx_m-\mz|)^2d\omega\bigg\}\right|,
  \end{multline}
  where
  \[\Lambda(x;\omega):=\omega\bigg(J_0(\omega|x|)^2+J_1(\omega|x|)^2\bigg).\]
\end{thm}
\begin{proof}
For the sake of simplicity, we assume that $M_f=M$ for $f=1,2,\cdots,F$, i.e., the difference $\lambda_1-\lambda_F$ is small enough\footnote{In the numerical experiments, we set $\lambda_1=0.6$ and $\lambda_F=0.3$, i.e., $\lambda_1-\lambda_F=0.3$ is small enough, refer to Section \ref{Sec3}.}. Then, based on the structure (\ref{StructureSingleSM}), we can say that
\begin{align*}
  \W_{\mathrm{MF}}(\mz;F)\approx&\frac{1}{F}\left|\sum_{f=1}^{F}\sum_{m=1}^{M}\left\{J_0(\omega_f|\mx_m-\mz|)^2-\left\langle\frac{\mx_m-\mz}{|\mx_m-\mz|},\mt(\mx_m)+\mn(\mx_m)\right\rangle^2J_1(\omega_f|\mx_m-\mz|)^2\right\}\right|\\
  \approx&\frac{1}{\omega_F-\omega_1}\left|\sum_{m=1}^{M}\bigg\{\int_{\omega_1}^{\omega_F}J_0(\omega|\mx_m-\mz|)^2d\omega-\left\langle\frac{\mx_m-\mz}{|\mx_m-\mz|},\mt(\mx_m)+\mn(\mx_m)\right\rangle^2\int_{\omega_1}^{\omega_F}J_1(\omega|\mx_m-\mz|)^2d\omega\bigg\}\right|.
\end{align*}

Since, following relation holds for $x\in\mathbb{R}$
  \[\int J_0(x)^2dx=x\bigg(J_0(x)^2+J_1(x)^2\bigg)+\int J_1(x)^2dx,\]
we can evaluate
\begin{multline*}
  \int_{\omega_1}^{\omega_F}J_0(\omega|\mx_m-\mz|)^2d\omega
  =\omega_F\bigg(J_0(\omega_F|\mx_m-\mz|)^2+J_1(\omega_F|\mx_m-\mz|)^2\bigg)\\
  -\omega_1\bigg(J_0(\omega_1|\mx_m-\mz|)^2+J_1(\omega_1|\mx_m-\mz|)^2\bigg)+\int_{\omega_1}^{\omega_F}J_1(\omega|\mx_m-\mz|)^2d\omega.
\end{multline*}
With this, we can obtain (\ref{StructureMultiSM}).
\end{proof}

Now, let us compare results in Theorems \ref{TheoremSingleSM} and \ref{TheoremMultiSM}. Based on structures (\ref{StructureSingleSM}) and (\ref{StructureMultiSM}), imaging functionals are composed with contributing and disturbing terms for imaging. First, contributing terms of (\ref{StructureSingleSM}) and (\ref{StructureMultiSM}) are
\[J_0(\omega|\mx_m-\mz|)\quad\mbox{and}\quad\frac{1}{\omega_F-\omega_1}\bigg(\Lambda(\mx_m-\mz;\omega_F)-\Lambda(\mx_m-\mz;\omega_1)\bigg),\]
respectively. Since $\Lambda(\mx_m-\mz;\omega_F)-\Lambda(\mx_m-\mz;\omega_1)$ oscillates less than $J_0(\omega|\mx_m-\mz|)$ (see \cite{P-TD1}), identified shape of the supporting curve via the map $\W_{\mathrm{MF}}(\mz;F)$ will be better than the one via the map $\W_{\mathrm{SF}}(\mz;\omega)$. Next, disturbing terms of (\ref{StructureSingleSM}) and (\ref{StructureMultiSM}) are
\[2\left\langle\frac{\mx_m-\mz}{|\mx_m-\mz|},\mn(\mx_m)\right\rangle^2J_1(\omega|\mx_m-\mz|)^2\quad\mbox{and}\quad\left(1-2\left\langle\frac{\mx_m-\mz}{|\mx_m-\mz|},\mn(\mx_m)\right\rangle^2\right)\int_{\omega_1}^{\omega_F}J_1(\omega|\mx_m-\mz|)^2d\omega,\]
respectively. Similar to the comparison of contributing terms, we can observe that since the term
\[\int_{\omega_1}^{\omega_F}J_1(\omega|\mx_m-\mz|)^2d\omega\] oscillates less than $J_1(\omega|\mx_m-\mz|)^2$ and the factor
\[1-2\left\langle\frac{\mx_m-\mz}{|\mx_m-\mz|},\mn(\mx_m)\right\rangle^2\]
will reduce magnitude of disturbing term, disturbing term of (\ref{StructureMultiSM}) will affect imaging performance less than the one of (\ref{StructureSingleSM}). Thus, we can conclude the following result. We believe that following fact can be proved on the basis of the Statistical Hypothesis Testing considered in \cite{AGKPS}.

\begin{cor}\label{CorollaryImprovement}
  Maps of multi-frequency subspace migration $\W_{\mathrm{MF}}(\mz;F)$ yields better imaging results owing to less oscillation than single-frequency one $\W_{\mathrm{SF}}(\mz;\omega)$. This means that unexpected artifacts in the map of $\W_{\mathrm{MF}}(\mz;F)$ are mitigated when $F$ is sufficiently large.
\end{cor}

Same as the single-frequency, we can immediately conclude that the following result of uniqueness holds.
\begin{cor}\label{CorollaryUnique2}
   Suppose that the values of $\omega_f$ are sufficiently high. If the total number $N$ of incident and observation directions and total number $F$ of applied frequencies are sufficiently large, then the shape of supporting curve $\sigma$ of thin inclusion $\Gamma$ can be obtained uniquely via the map of $\W_{\mathrm{MF}}(\mz;F)$.
\end{cor}

\begin{rem}[Filtering]\label{RemarkFiltering}
  Similar to the case of single-frequency, let us consider the method of filtering. Since
  \[\max_{\mz\in\mathbb{R}^2}\left|1-2\left\langle\frac{\mx_m-\mz}{|\mx_m-\mz|},\mn(\mx_m)\right\rangle^2\right|=1,\]
  and
  \[\frac{1}{F}\sum_{f=1}^{F}J_1(\omega_f|\mz-\mx_m|)^2\approx\frac{1}{\omega_F-\omega_1}\int_{\omega_1}^{\omega_F}J_1(\omega|\mz-\mx_m|)^2d\omega\leq0.338567\approx0.340\]
  Hence, for multi-frequency imaging, we can define a filtering function such that
  \[\mathcal{F}_\mathrm{M}^{(1)}[x]=\left\{\begin{array}{ccc}
\medskip x & \mbox{if} & 0.340\leq x\leq1 \\
0 & \mbox{if} & 0\leq x<0.340.
\end{array}
\right.\]
Then, $\mathcal{F}_\mathrm{M}^{(1)}[\W_{\mathrm{MF}}(\mz;F)]$ will be an improved version of $\W_{\mathrm{MF}}(\mz;F)$.
\end{rem}

%\begin{figure}[h]
%\begin{center}
%\includegraphics[width=0.99\textwidth]{Bessel12M.eps}
%\caption{\label{Bessel12M}Graphs of $\sum_{f}J_0(\omega_f x)^2$ and $2\sum_{f}J_1(\omega_f x)^2$ for $\omega_1=2\pi/0.7$ and $\omega_F=2\pi/0.3$.}
%\end{center}
%\end{figure}

\begin{rem}\label{Remark2}
  Similar to the Remark \ref{Remark1}, in this paper, we consider the following normalized value
  \[\hat{\W}_{\mathrm{MF}}(\mz;F):=\frac{\W_{\mathrm{MF}}(\mz;F)}{\displaystyle\max_{\mz\in\mathbb{R}^2}|\W_{\mathrm{SF}}(\mz;F)|}\]
  and consider the filtered map $\mathcal{F}[\hat{\W}_{\mathrm{MF}}(\mz;F)]$ instead of $\W_{\mathrm{MF}}(\mz;F)$.
\end{rem}

\section{Results of numerical simulations}\label{Sec3}
In this section, we exhibit some results of numerical simulations to support Theorems \ref{TheoremSingleSM} and \ref{TheoremMultiSM}. In order to describe thin inclusions $\Gamma_j$, $j=1,2$, two supporting smooth curves are selected as follows:
\begin{align*}
  \sigma_1&=\set{[s-0.2,-0.5s^2+0.5]^T:-0.5\leq s\leq0.5}\\
  \sigma_2&=\set{[s+0.2,s^3+s^2-0.6]^T:-0.5\leq s\leq0.5}.
\end{align*}
The thickness $h$ of thin inclusions $\Gamma_j$ is equally set to $0.015$. We denote $\eps_j$ and $\mu_j$ be the permittivity and permeability of $\Gamma_j$, respectively, and set parameters $\mu_j$, $\mu_0$, $\eps_j$ and $\eps_0$ are $5,1,5$ and $1$, respectively. Since $\mu_0$ and $\eps_0$ are set to unity, the applied frequencies reads as $\omega_f=2\pi/\lambda_f$ at wavelength $\lambda_f$ for $f=1,2,\cdots,F(=10)$, which will be varied in the numerical examples between $\lambda_1=0.6$ and $\lambda_{10}=0.3$. For single frequency imaging, $\omega=2\pi/0.4$ is applied. In order to show robustness,  $10$dB Gaussian random noise is added to the unperturbed data.

First, let us examine the effect of the selection $\mc_n$ of (\ref{VecT}). Figure \ref{Gamma1V} shows maps of $\W_{\mathrm{SM}}(\mz;\omega)$ for $\mc_n=[1,1,0]^T$, $\mc_n=[1,\frac{1}{\sqrt{2}},\frac{1}{\sqrt{2}}]^T$, and $\mc_n=[1,\frac{3}{2},\frac{1}{2}]^T$ when the thin inhomogeneity is $\Gamma_1$. Throughout the result, we can observe that one cannot identify the shape of $\Gamma_1$ at this moment. Note that based on \ref{P1}, the dominant  eigenvectors are $\mn(\mx_m)$, $\mc_n$ must be of the form $[1,\mc(\mx_m)^T]^T$. Hence, from now on, we apply $\mc_n=[1,0,1]^T$ for $\W_{\mathrm{SM}}(\mz;\omega)$.

\begin{figure}[h]
\begin{center}
\includegraphics[width=0.325\textwidth]{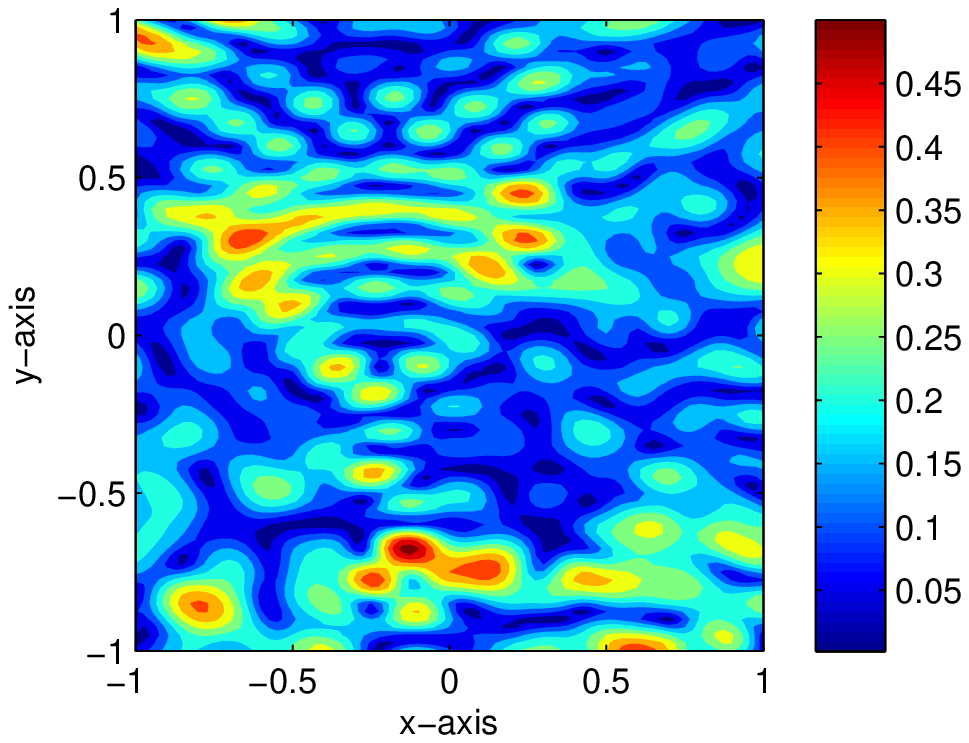}
\includegraphics[width=0.325\textwidth]{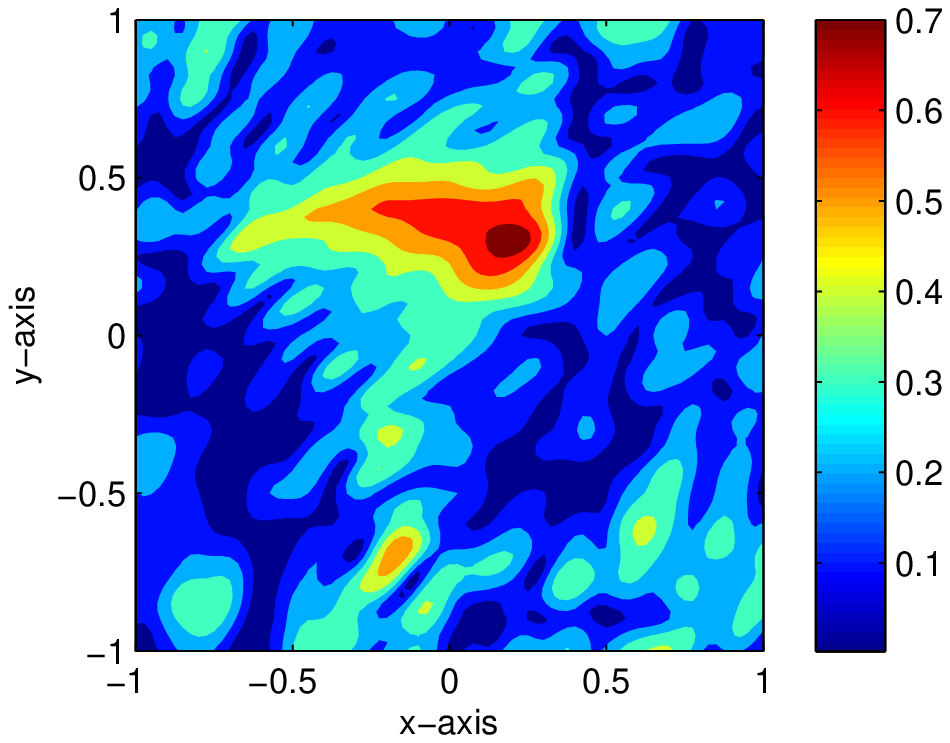}
\includegraphics[width=0.325\textwidth]{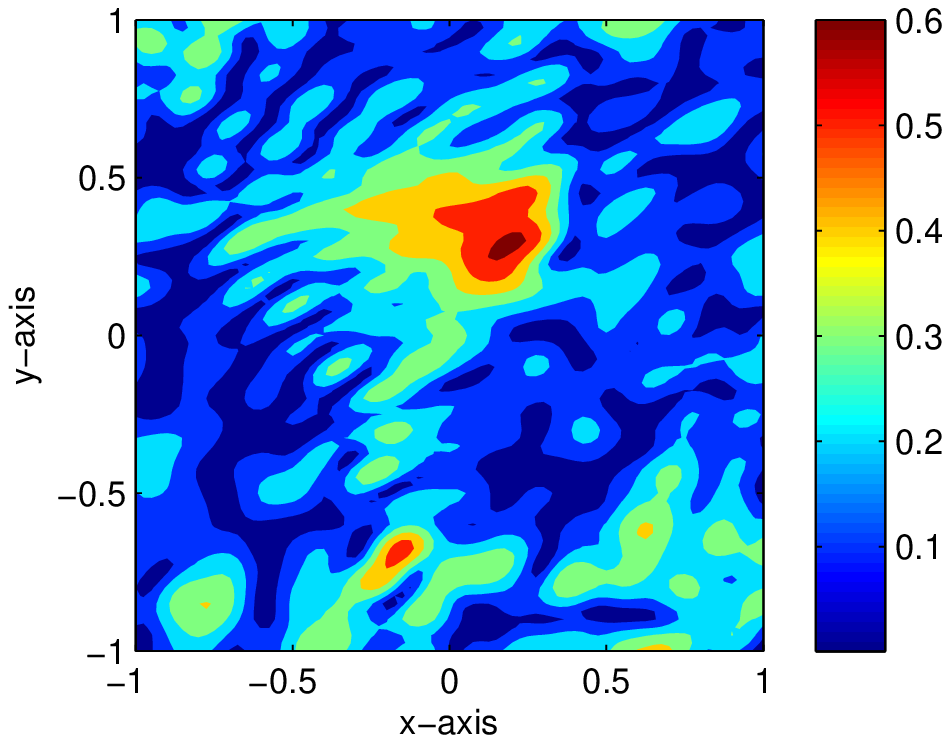}
\caption{\label{Gamma1V}Maps of $\W_{\mathrm{SM}}(\mz;\omega)$ for $\mc_n=[1,1,0]^T$ (left), $\mc_n=[1,\frac{1}{\sqrt{2}},\frac{1}{\sqrt{2}}]^T$ (center), and $\mc_n=[1,\frac{3}{2},\frac{1}{2}]^T$ (right) when the thin inhomogeneity is $\Gamma_1$.}
\end{center}
\end{figure}

Now, let us comsider the imaging results of $\W_{\mathrm{SM}}(\mz;\omega)$, $\W_{\mathrm{SF}}(\mz;\omega)$, and $\mathcal{F}[\hat{\W}_{\mathrm{SF}}(\mz;\omega)]$. On the basis of results in Figure \ref{Gamma1S}, we can observe that the shape of $\Gamma_1$ can be recognized via the maps of $\W_{\mathrm{SM}}(\mz;\omega)$ and $\W_{\mathrm{SF}}(\mz;\omega)$ but the imaging seems rather coarse for the traditional method $\W_{\mathrm{SM}}(\mz;\omega)$, better for the proposed one $\W_{\mathrm{SF}}(\mz;\omega)$. Furthermore, $\mathcal{F}[\hat{\W}_{\mathrm{SF}}(\mz;\omega)]$ exhibits very accurate shape of $\Gamma_1$ so that suggested filtering method seems very effective.

\begin{figure}[h]
\begin{center}
\includegraphics[width=0.325\textwidth]{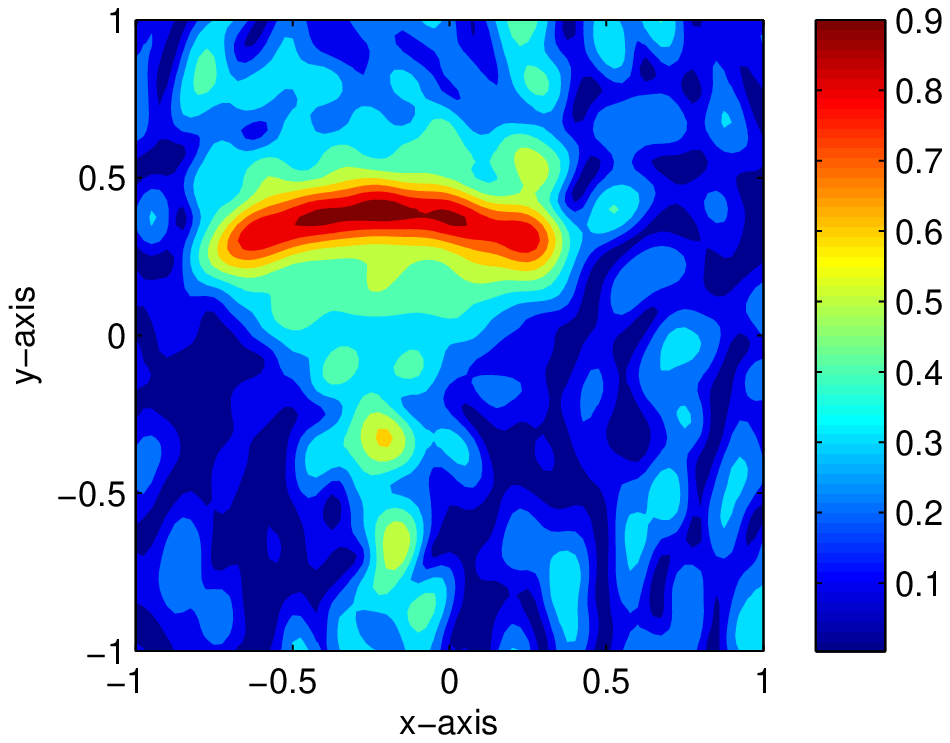}
\includegraphics[width=0.325\textwidth]{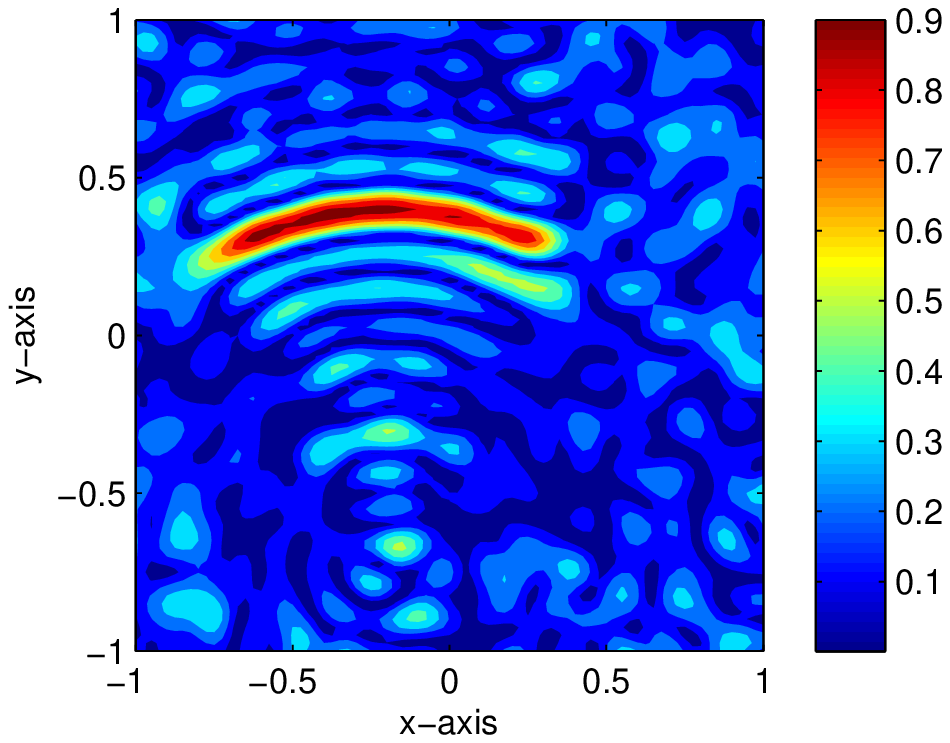}
\includegraphics[width=0.325\textwidth]{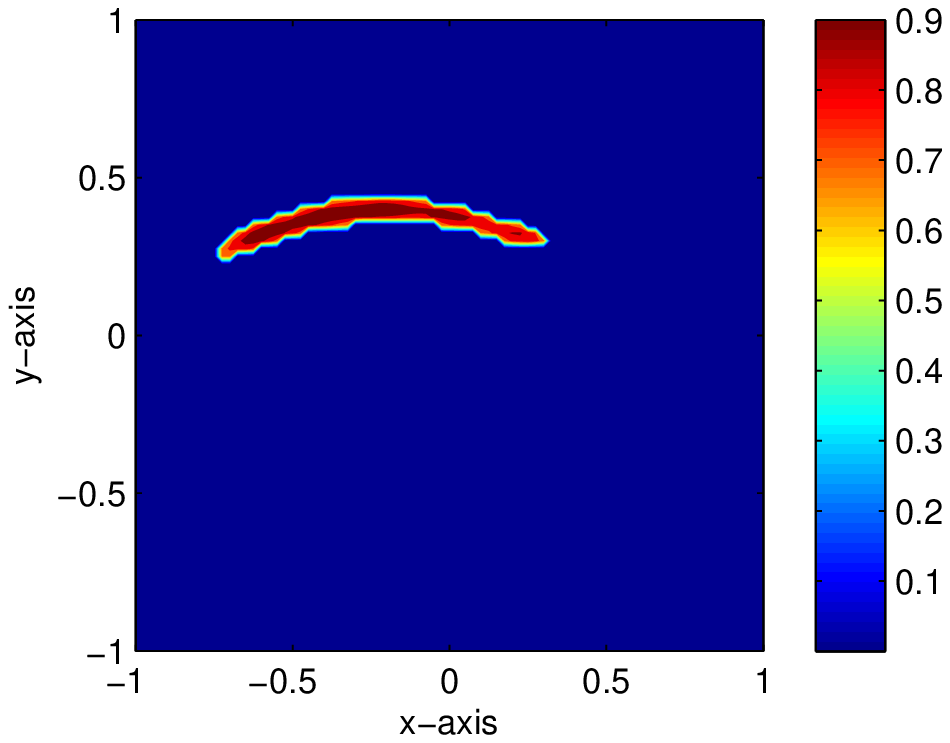}
\caption{\label{Gamma1S}Maps of $\W_{\mathrm{SM}}(\mz;\omega)$ (left), $\W_{\mathrm{SF}}(\mz;\omega)$ (center), and $\mathcal{F}[\hat{\W}_{\mathrm{SF}}(\mz;\omega)]$ (right) when the thin inhomogeneity is $\Gamma_1$.}
\end{center}
\end{figure}

Maps of $\W_{\mathrm{SM}}(\mz;\omega)$, $\W_{\mathrm{SF}}(\mz;\omega)$, and $\mathcal{F}[\hat{\W}_{\mathrm{SF}}(\mz;\omega)]$ are shown in Figure \ref{Gamma2S} when the thin inhomogeneity is $\Gamma_2$. Similar to the results in Figure \ref{Gamma1S}, the shape of $\Gamma_1$ can be recognized but obtained shape of $\Gamma_1$ via $\W_{\mathrm{SF}}(\mz;\omega)$ looks like an \textit{anchor} due to the appearance of unexpected artifacts. Although, some artifacts are still visible, the result via $\W_{\mathrm{SF}}(\mz;\omega)$ seems better than the one via $\W_{\mathrm{SM}}(\mz;\omega)$. And, similar to the previous result, filtering method seems still effective.

\begin{figure}[h]
\begin{center}
\includegraphics[width=0.325\textwidth]{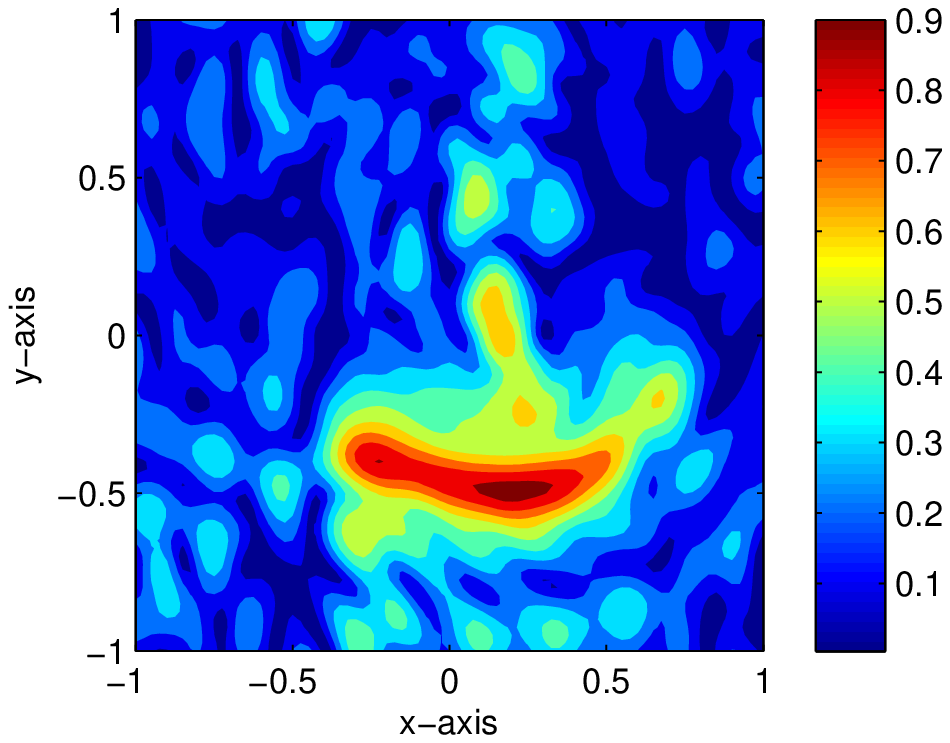}
\includegraphics[width=0.325\textwidth]{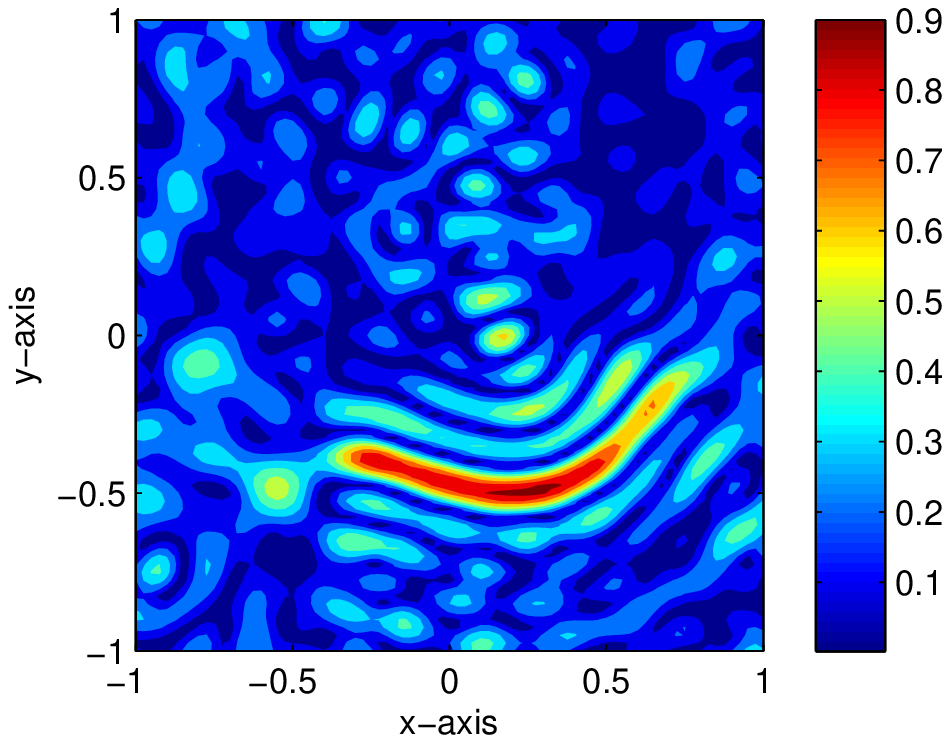}
\includegraphics[width=0.325\textwidth]{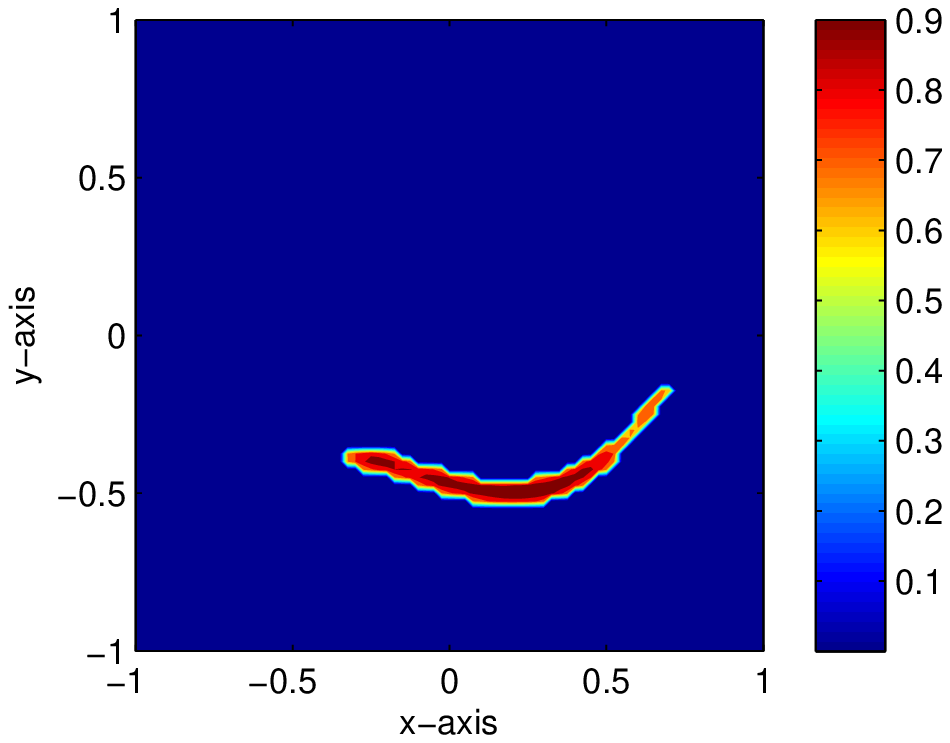}
\caption{\label{Gamma2S}Same as Figure \ref{Gamma1S} except the thin inhomogeneity is $\Gamma_2$.}
\end{center}
\end{figure}

It is well-known that one of advantage of subspace migration is its straightforward application to the imaging of multiple inhomogeneities. Figure \ref{GammaM1} shows the maps of $\W_{\mathrm{SM}}(\mz;\omega)$, $\W_{\mathrm{SF}}(\mz;\omega)$, and $\mathcal{F}[\hat{\W}_{\mathrm{SF}}(\mz;\omega)]$ for imaging multiple thin inhomogeneities $\Gamma_1\cup\Gamma_2$ with same permittivity $\eps_1=\eps_2=5$ and permeability $\mu_1=\mu_2=5$. Similar to the imaging of single inhomogeneity, we can observe that $\W_{\mathrm{SF}}(\mz;\omega)$ is an effective method. However, due to the appearance of artifacts, it is hard to identify true shape of inhomogeneities. So, in contrast to the imaging of single inhomogeneity, filtering method is not effective for imaging of multiple inhomogeneities.

\begin{figure}[h]
\begin{center}
\includegraphics[width=0.325\textwidth]{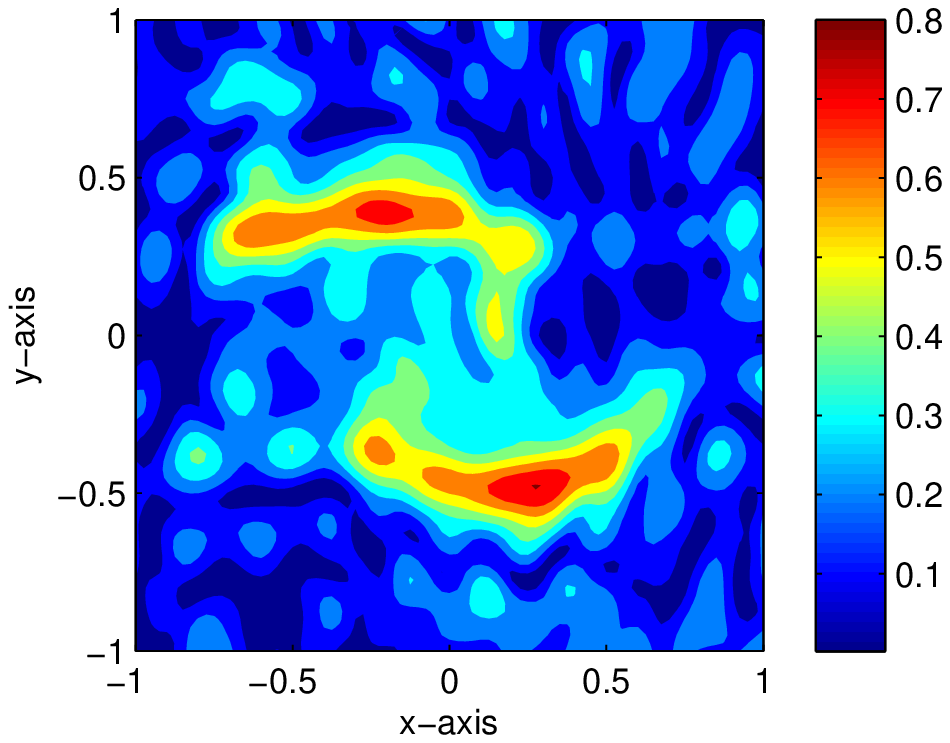}
\includegraphics[width=0.325\textwidth]{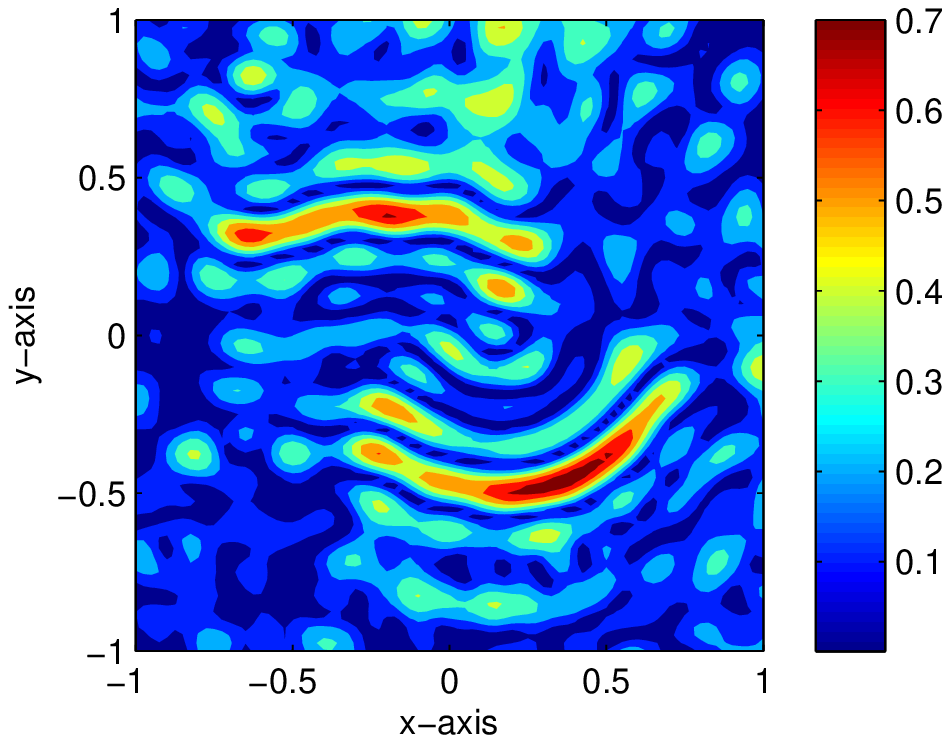}
\includegraphics[width=0.325\textwidth]{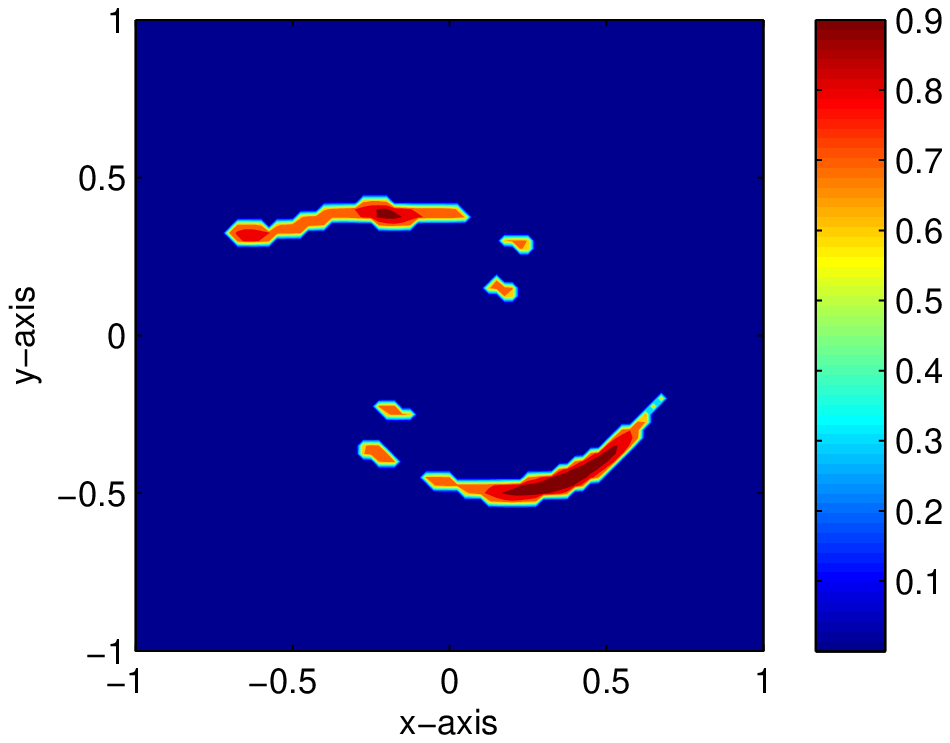}
\caption{\label{GammaM1}Same as Figure \ref{Gamma1S} except with same permittivities and permeabilities when the thin inclusion is $\Gamma_1\cup\Gamma_2$.}
\end{center}
\end{figure}

Figure \ref{GammaM2} shows the maps of $\W_{\mathrm{SM}}(\mz;\omega)$, $\W_{\mathrm{SF}}(\mz;\omega)$, and $\mathcal{F}[\hat{\W}_{\mathrm{SF}}(\mz;\omega)]$ under the same configuration as the previous result in Figure \ref{GammaM1}, except for different material
properties, $\eps_1=\mu_1=10$ and $\eps_2=\mu_2=5$. Then, based on the results in \cite{PL3,P-SUB1}, values of $\W_{\mathrm{SM}}(\mz;\omega)$ and $\W_{\mathrm{SF}}(\mz;\omega)$ for $\mz\in\Gamma_2$ will be smaller than $\W_{\mathrm{SM}}(\mz;\omega)$ and $\W_{\mathrm{SF}}(\mz;\omega)$ for $\mz\in\Gamma_1$, respectively. Furthermore, due to the unexpected artifacts, the shape of $\Gamma_2$ cannot be identified, while $\Gamma_1$ can be identified. Correspondingly, only the shape of $\Gamma_1$ can be identified in the map of $\mathcal{F}[\hat{\W}_{\mathrm{SF}}(\mz;\omega)]$.

\begin{figure}[h]
\begin{center}
\includegraphics[width=0.325\textwidth]{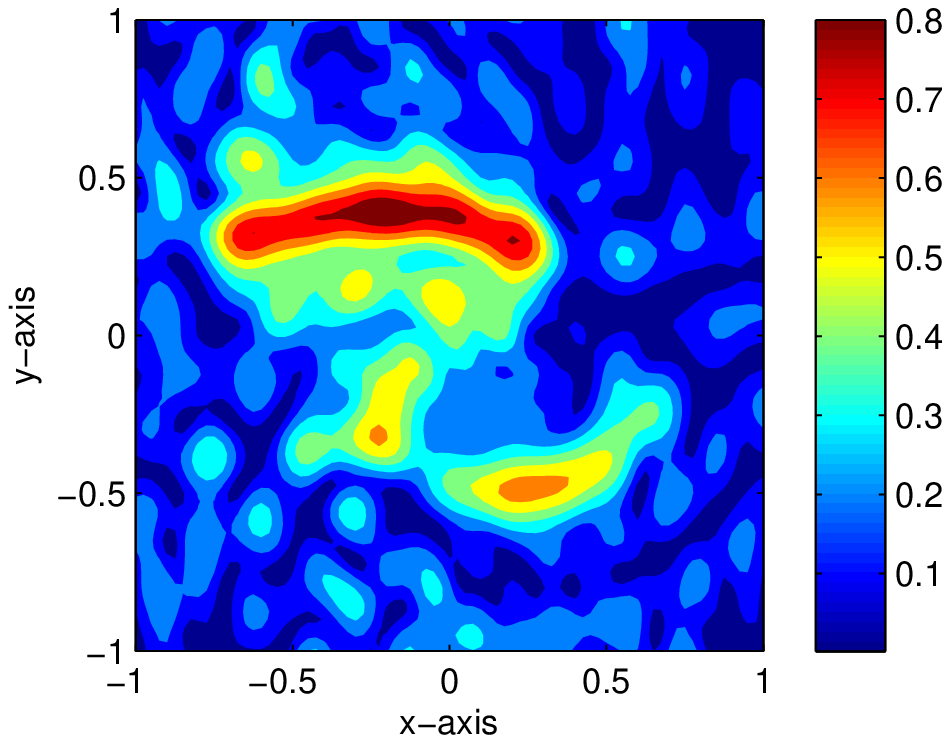}
\includegraphics[width=0.325\textwidth]{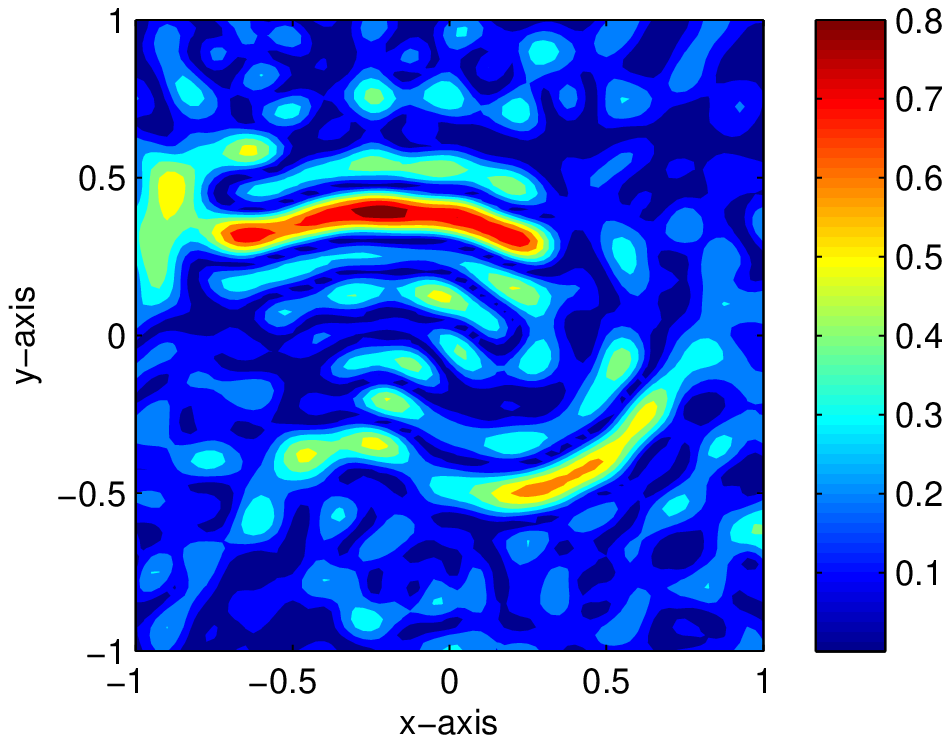}
\includegraphics[width=0.325\textwidth]{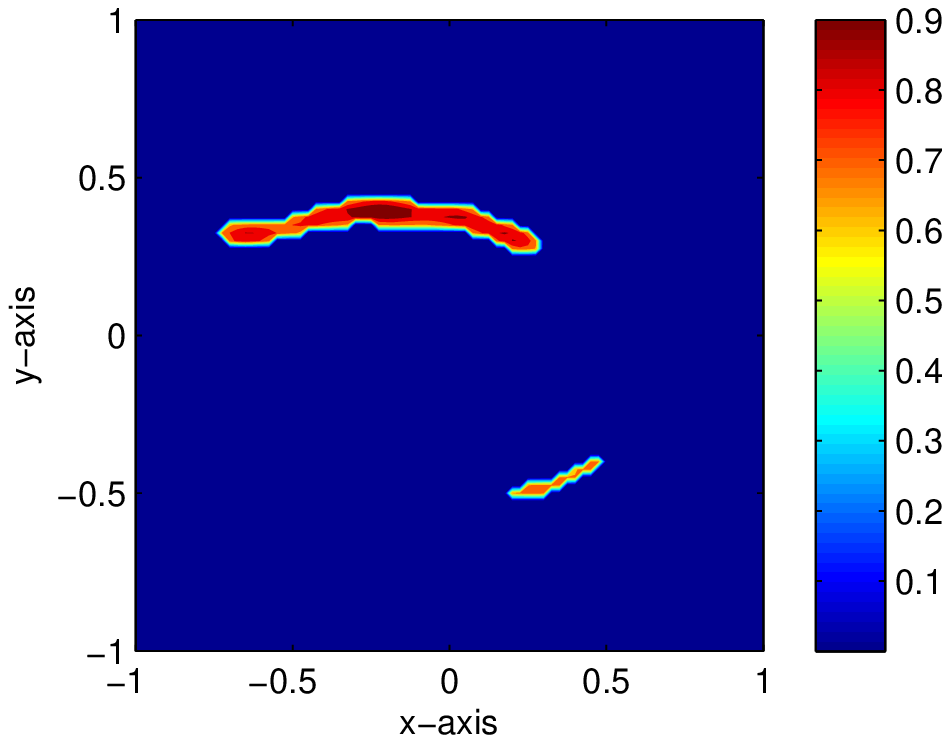}
\caption{\label{GammaM2}Same as Figure \ref{GammaM1} except with different permittivities and permeabilities.}
\end{center}
\end{figure}

From now on, we consider the multi-frequency imaging. Maps of $\W_{\mathrm{MM}}(\mz;F)$, $\W_{\mathrm{MF}}(\mz;F)$, and $\mathcal{F}[\hat{\W}_{\mathrm{MF}}(\mz;F)]$ are exhibited in Figure \ref{GammaMulti1} when the thin inhomogeneity is $\Gamma_1$. Although, map of $\W_{\mathrm{MF}}(\mz;F)$ contains more artifacts than $\W_{\mathrm{MM}}(\mz;F)$, identified shape via the map of $\W_{\mathrm{MF}}(\mz;F)$ seems close to the true shape of $\Gamma_1$. It is interesting to observe that opposite to the Remark \ref{RemarkFiltering}, unexpected peak of small (but cannot be negligible) magnitude still remaining in the neighborhood of the tip of $\Gamma_1$ but the result seems very nice. Similar phenomenon can be examined through the results in Figure \ref{GammaMulti2} when the thin inhomogeneity is $\Gamma_2$.

\begin{figure}[h]
\begin{center}
\includegraphics[width=0.325\textwidth]{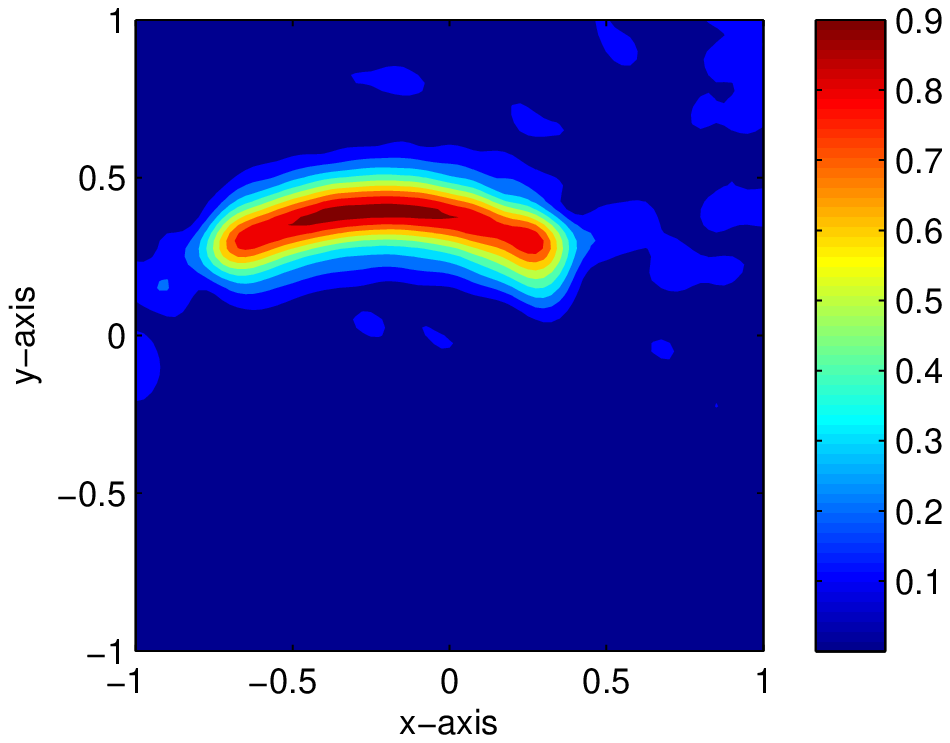}
\includegraphics[width=0.325\textwidth]{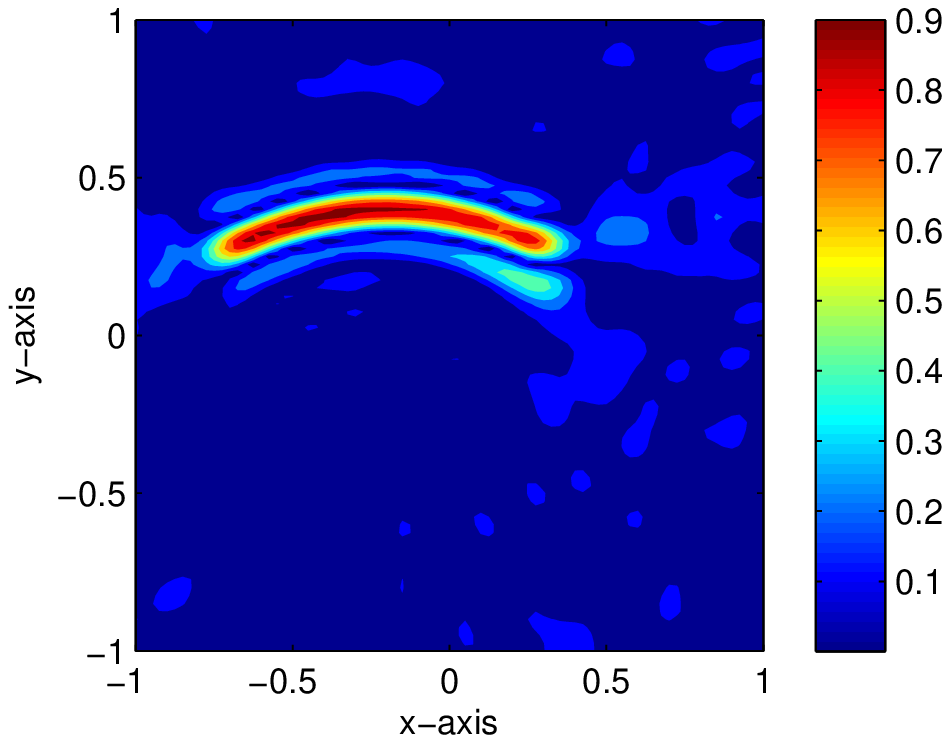}
\includegraphics[width=0.325\textwidth]{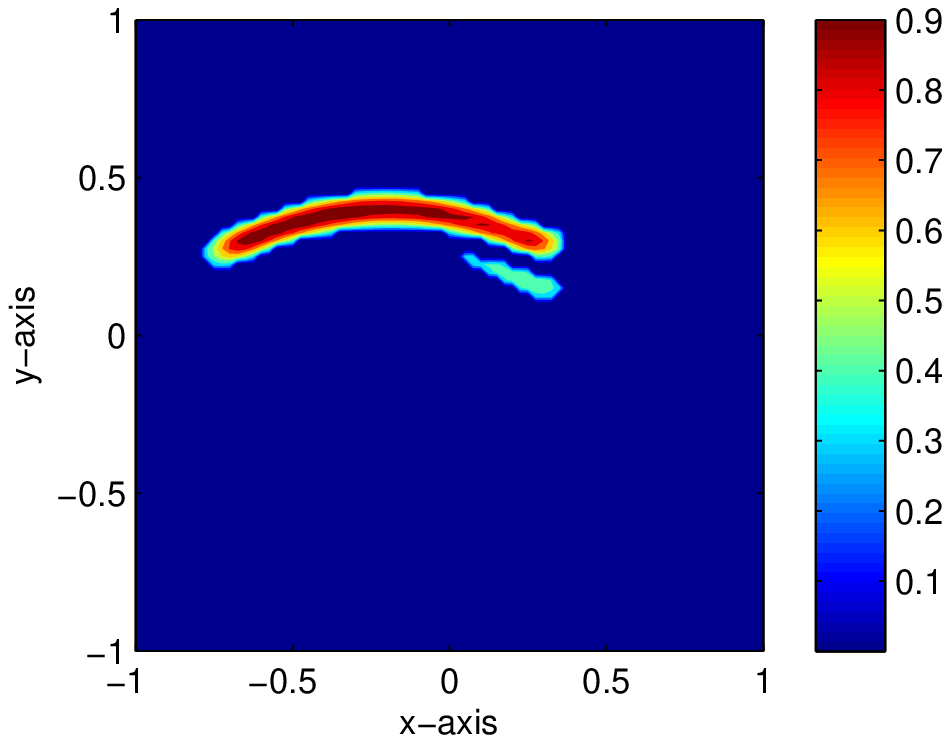}
\caption{\label{GammaMulti1}Maps of $\W_{\mathrm{MM}}(\mz;F)$ (left), $\W_{\mathrm{MF}}(\mz;F)$ (center), and $\mathcal{F}[\hat{\W}_{\mathrm{MF}}(\mz;F)]$ (right) when the thin inhomogeneity is $\Gamma_1$.}
\end{center}
\end{figure}

\begin{figure}[h]
\begin{center}
\includegraphics[width=0.325\textwidth]{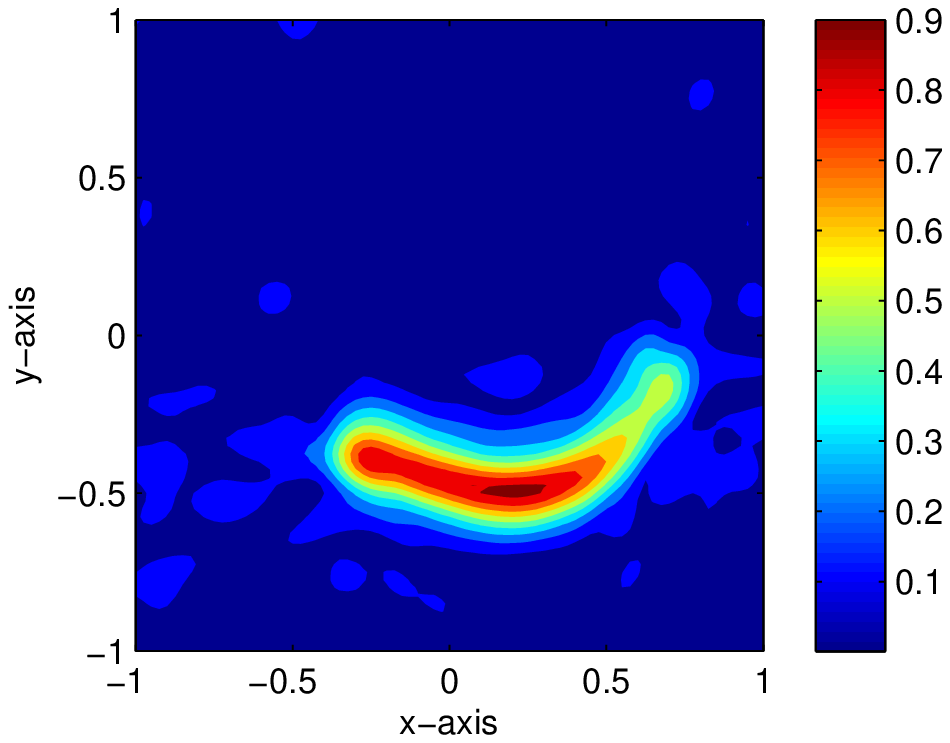}
\includegraphics[width=0.325\textwidth]{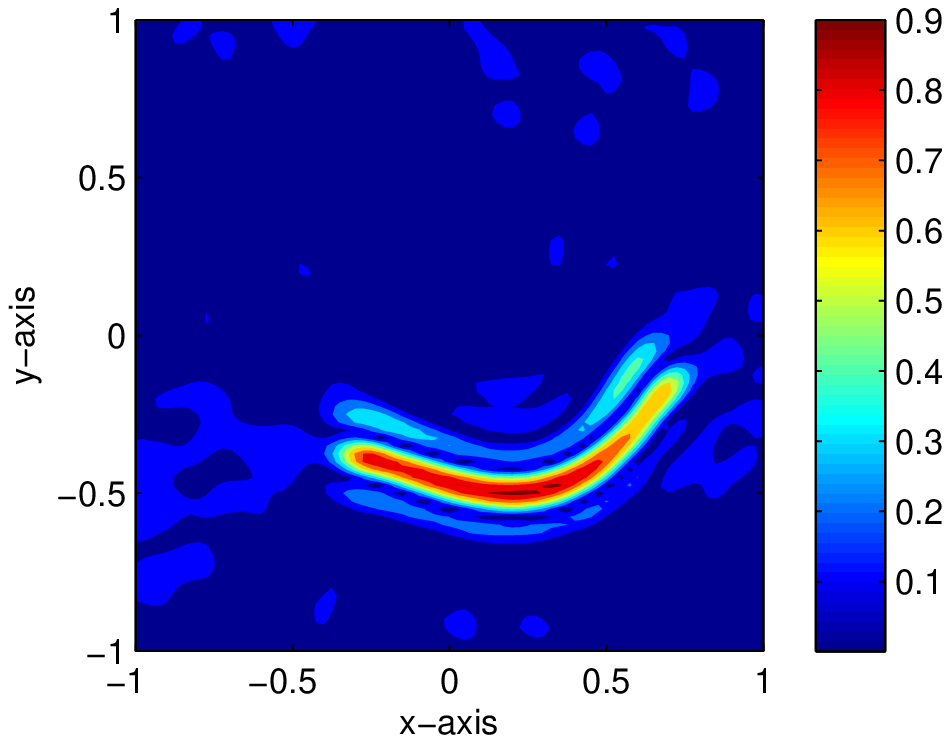}
\includegraphics[width=0.325\textwidth]{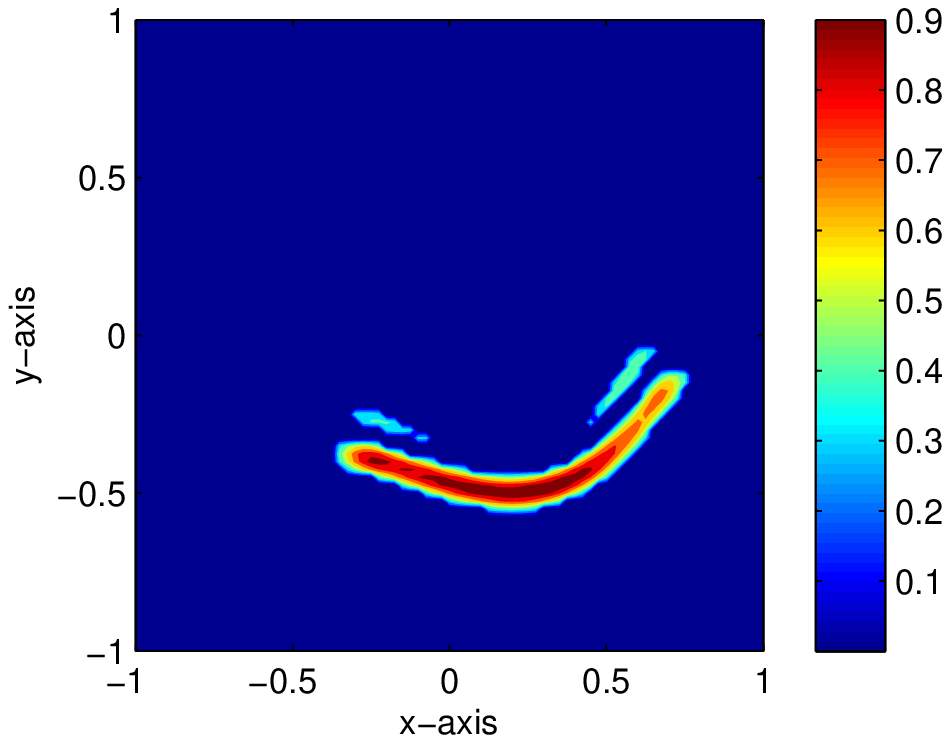}
\caption{\label{GammaMulti2}Same as Figure \ref{GammaMulti1} except the thin inhomogeneity is $\Gamma_2$.}
\end{center}
\end{figure}

Figure \ref{GammaMultiM1} shows the maps of $\W_{\mathrm{MM}}(\mz;F)$, $\W_{\mathrm{MF}}(\mz;F)$, and $\mathcal{F}[\hat{\W}_{\mathrm{MF}}(\mz;F)]$ for imaging multiple thin inhomogeneities $\Gamma_1\cup\Gamma_2$ with same permittivity $\eps_1=\eps_2=5$ and permeability $\mu_1=\mu_2=5$. Similar to the single-frequency imaging, it seems that $\W_{\mathrm{MF}}(\mz;F)$ performs better imaging accomplishment than $\W_{\mathrm{MM}}(\mz;F)$. However, due to the appearance of artifacts, it is hard to identify true shape of inhomogeneities. Furthermore, in contrast to the single-frequency imaging, filtering method seems very effective for imaging although some peaks of small magnitudes are remaining.

\begin{figure}[h]
\begin{center}
\includegraphics[width=0.325\textwidth]{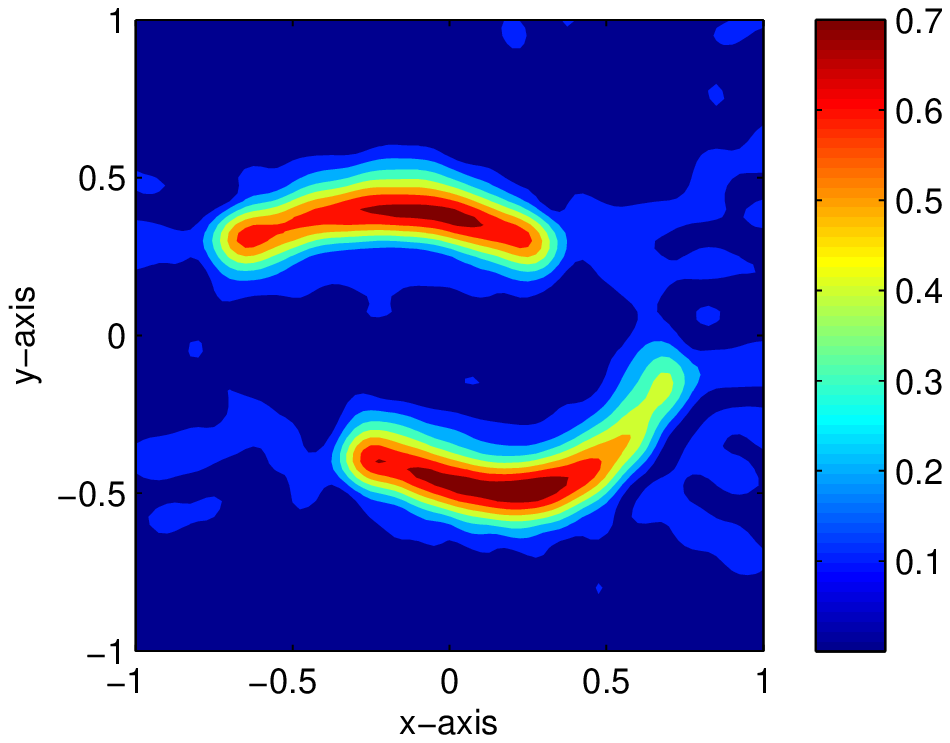}
\includegraphics[width=0.325\textwidth]{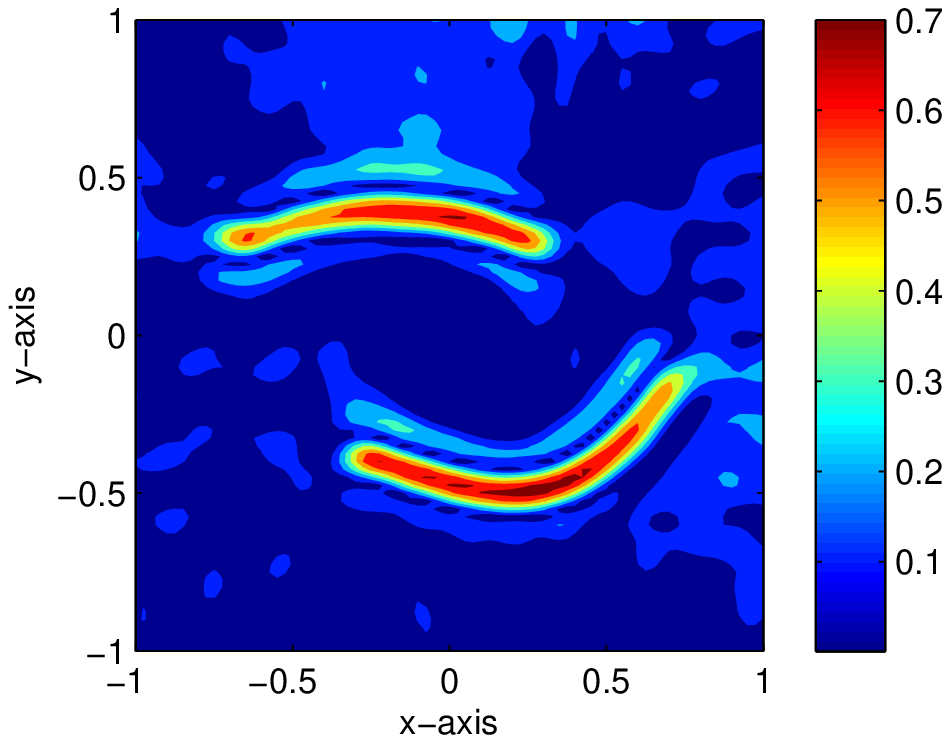}
\includegraphics[width=0.325\textwidth]{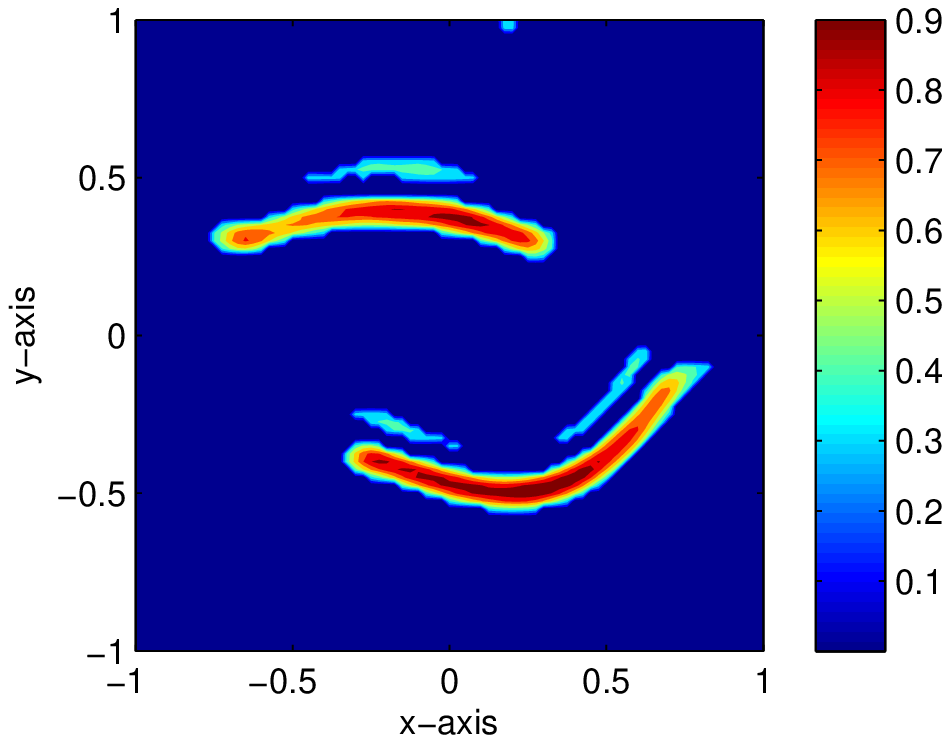}
\caption{\label{GammaMultiM1}Same as Figure \ref{GammaMulti1} except with same permittivities and permeabilities when the thin inclusion is $\Gamma_1\cup\Gamma_2$.}
\end{center}
\end{figure}

Figure \ref{GammaM2} shows the maps of $\W_{\mathrm{MM}}(\mz;F)$, $\W_{\mathrm{MF}}(\mz;F)$, and $\mathcal{F}[\hat{\W}_{\mathrm{MF}}(\mz;F)]$ under the same configuration as the previous result in Figure \ref{GammaMultiM1}, except for different material
properties, $\eps_1=\mu_1=10$ and $\eps_2=\mu_2=5$. Similar to the previous example, $\W_{\mathrm{MF}}(\mz;F)$ can be regarded as an improved version of $\W_{\mathrm{MM}}(\mz;F)$.

\begin{figure}[h]
\begin{center}
\includegraphics[width=0.325\textwidth]{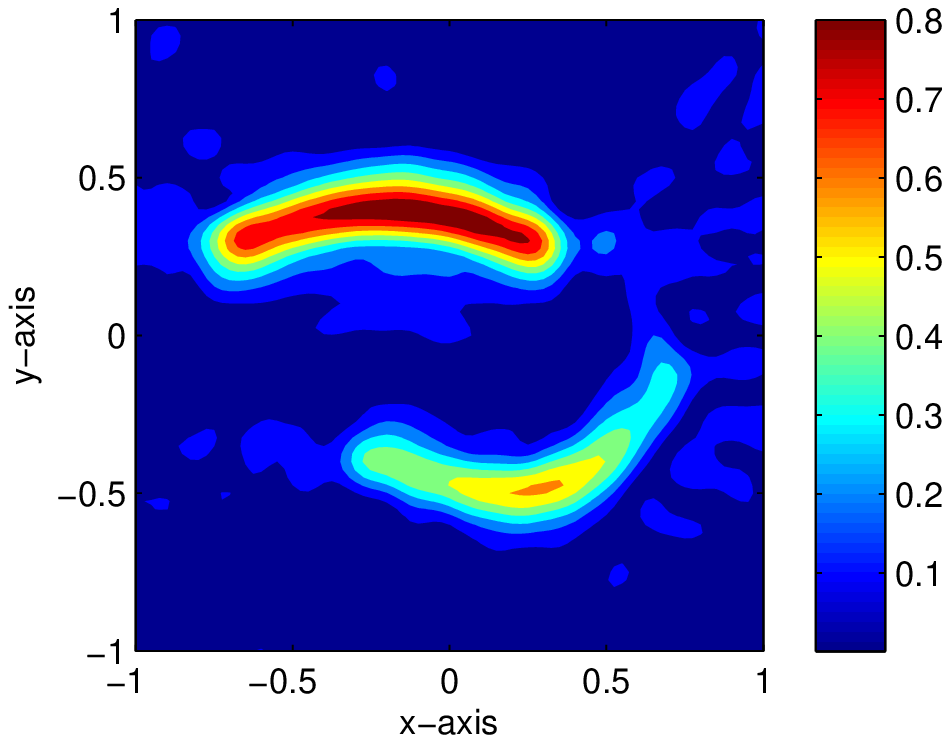}
\includegraphics[width=0.325\textwidth]{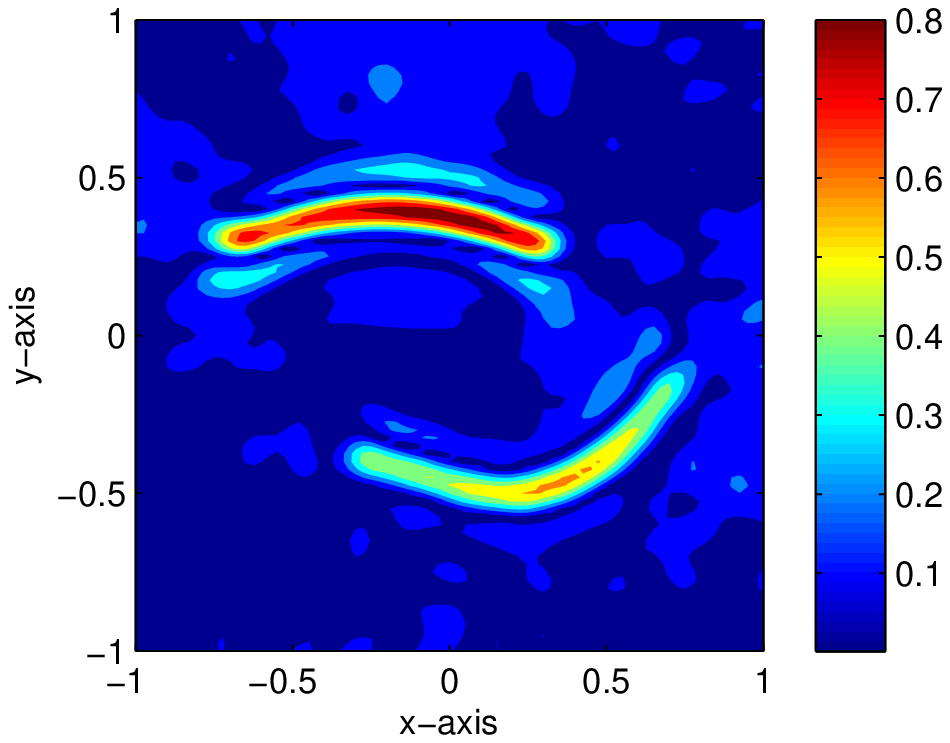}
\includegraphics[width=0.325\textwidth]{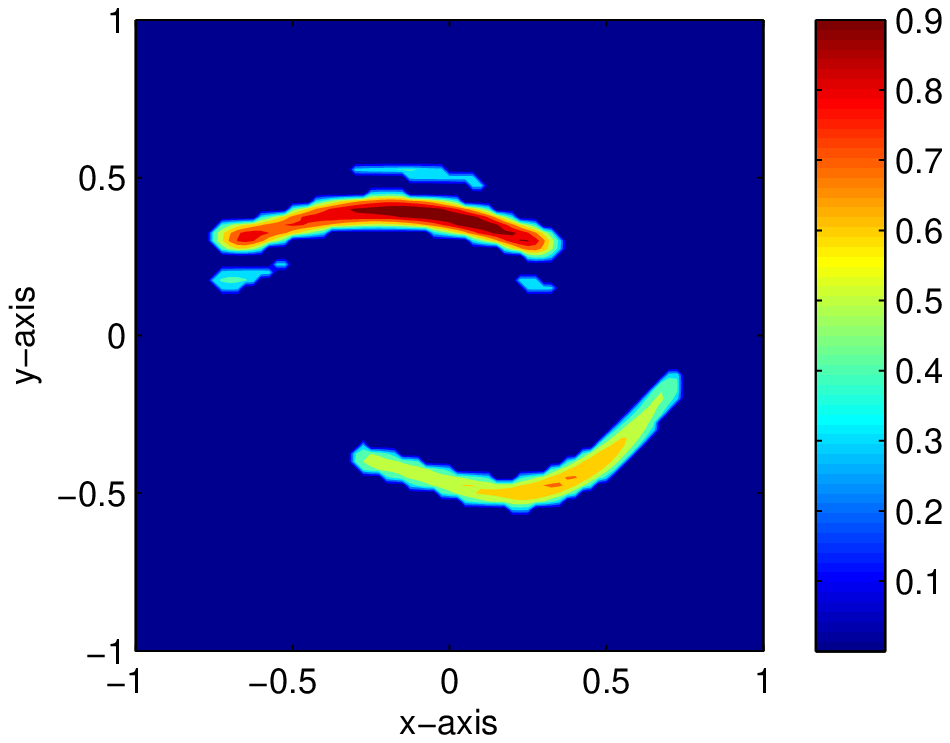}
\caption{\label{GammaMultiM2}Same as Figure \ref{GammaMultiM1} except with different permittivities and permeabilities.}
\end{center}
\end{figure}

For the final example, let us consider the application of filtering for imaging of multiple inhomogeneities when their permittivities and permeabilities are different to each other. A simply way is to divide search domain into two(or more)-disjoint areas and applying filtering function to each areas. Figure \ref{GammaMultiD} shows corresponding results. By comparing the results in Figures \ref{GammaM2} and \ref{GammaMulti2}, identified shapes are more accurate than traditional ones.

\begin{figure}[h]
\begin{center}
\includegraphics[width=0.325\textwidth]{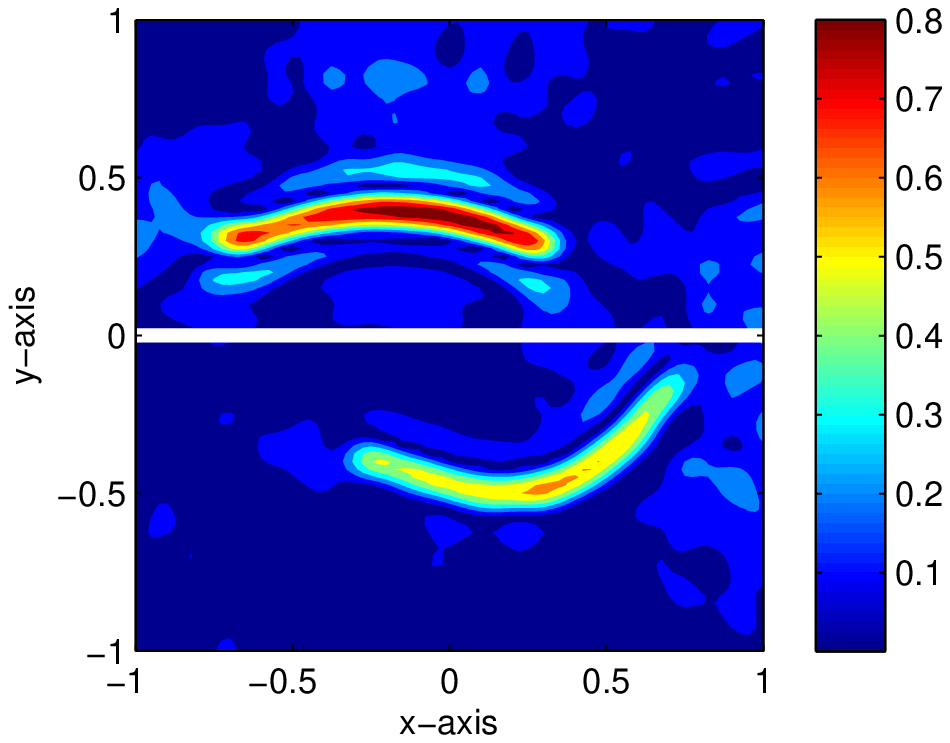}
\includegraphics[width=0.325\textwidth]{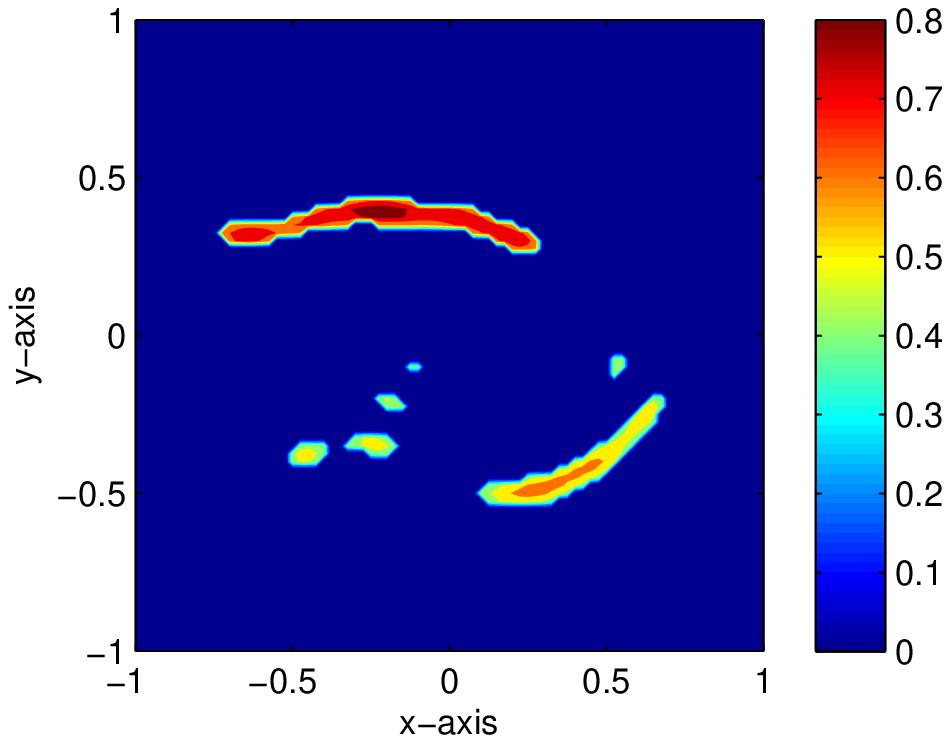}
\includegraphics[width=0.325\textwidth]{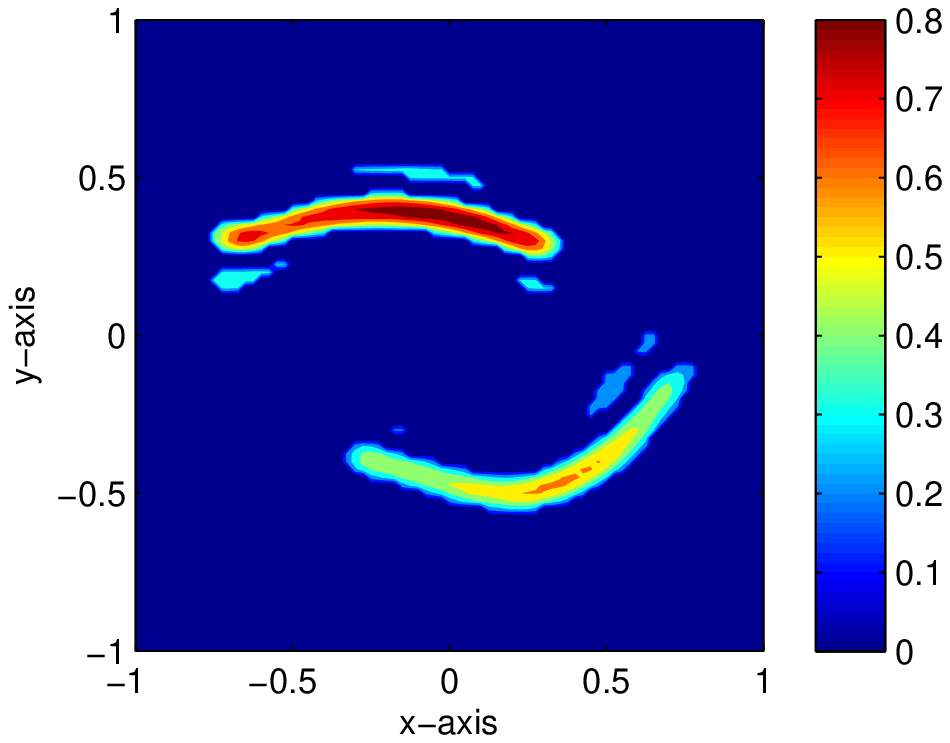}
\caption{\label{GammaMultiD}Divided search domain into two-disjoint areas (left), map of $\mathcal{F}[\hat{\W}_{\mathrm{SF}}(\mz;\omega)]$ (center), and map of $\mathcal{F}[\hat{\W}_{\mathrm{MF}}(\mz;F)]$ (right) when the thin inhomogeneities are $\Gamma_1\cup\Gamma_2$.}
\end{center}
\end{figure}

\section{Conclusion}\label{Sec4}
In this paper, we considered the subspace migration imaging functional without any \textit{a priori} information of thin inhomogeneities. We derived a relationship between the imaging functional and the Bessel functions of integer order of the first kind. The derived results indicated that although some unexpected artifacts still appear, proposed subspace migration improves traditional one. For a further improvement, two-different methodologies are also suggested and successfully applied.

Here, we focused on the imaging of thin, curve-like electromagnetic inhomogeneities. In the same line of though, the analysis of subspace migration for the imaging of perfectly conducting cracks in Transverse Magnetic (TM) and Transverse Electric (TE) cases will an interesting research topic.

Finally, we have been considering an imaging of a two-dimensional thin electromagnetic inclusions. The analysis could be extended to a three-dimensional problem; refer to \cite{AILP,IGLP,SCC} for related work.

\section*{Acknowledgement}
The author would like to acknowledge Dominique Lesselier for his precious comments. Parts of current work was done while the author was visiting Laboratoire des Signaux et Syst{\`e}ms (L2S), Ecole Sup{\'e}rieure d'Electricit{\'e} (Sup{\'e}lec). This research was supported by Basic Science Research Program through the National Research Foundation of Korea (NRF) funded by the Ministry of Education(No. NRF-2014R1A1A2055225) and the research program of Kookmin University in Korea.

\appendix
\section{Orthogonality of $\mH_m^{(s)}$, $s=1,2,3$, and their norms}\label{secA}
Note that since
\[\left\langle\mH_m^{(1)}(\omega),\mH_m^{(2)}(\omega)\right\rangle=\sum_{n=1}^{N}\langle\sqrt2\vt_n,\mt(\mx_m)\rangle=\frac{N}{\sqrt{2}\pi}\int_{\mathbb{S}^1}\langle\vt,\mt(\mx_m)\rangle d\vt=\sqrt2\int_0^{2\pi}\cos(\theta+\varphi_m)d\theta=0,\]
$\mH_m^{(1)}(\omega)$ is orthogonal to $\mH_m^{(2)}(\omega)$ and similarly, orthogonal to $\mH_m^{(3)}(\omega)$. Furthermore, if $\mt(\mx_m)=[\cos\varphi_m,\sin\varphi_m]^T$ then since $\mn(\mx_m)=[-\sin\varphi_m,\cos\varphi_m]^T$,
\begin{align*}
  \left\langle\mH_m^{(2)}(\omega),\mH_m^{(3)}(\omega)\right\rangle&=\sum_{n=1}^{N}\langle\sqrt2\vt_n,\mt(\mx_m)\rangle\langle\sqrt2\vt_n,\mn(\mx_m)\rangle=\frac{N}{\pi}\int_{\mathbb{S}^1}\langle\vt,\mt(\mx_m)\rangle\langle\vt,\mn(\mx_m)\rangle d\vt\\
  &=\frac{N}{\pi}\int_0^{2\pi}\cos(\theta-\varphi_m)\sin(\theta-\varphi_m)d\theta=\frac{N}{2\pi}\int_0^{2\pi}\sin(2\theta-2\varphi_m)d\theta=0.
\end{align*}
Hence $\mH_m^{(2)}(\omega)$ is orthogonal to $\mH_m^{(3)}(\omega)$.

Based on (\ref{Vec1}), it is easy to observe that
\[||\mH_m^{(1)}(\omega)||^2=\left\langle\mH_m^{(1)}(\omega),\mH_m^{(1)}(\omega)\right\rangle=N.\]
Since $N$ is sufficiently large,
\[||\mH_m^{(2)}(\omega)||^2=\left\langle\mH_m^{(2)}(\omega),\mH_m^{(2)}(\omega)\right\rangle=\sum_{n=1}^{N}\langle\vt_n,\mt(\mx_m)\rangle^2=\frac{2N}{\pi}\int_{\mathbb{S}^1}\langle\vt,\mt(\mx_m)\rangle^2d\vt.\]
Let us consider the polar coordinate $\vt=[\cos\theta,\sin\theta]^T$, $\mt(\mx_m)=[\cos\varphi_m,\sin\varphi_m]^T$. Then, performing an elementary calculus yields
\[\int_{\mathbb{S}^1}\langle\vt,\mt(\mx_m)\rangle^2d\vt=\int_0^{2\pi}\cos^2(\theta-\varphi_m)d\theta=\left[\frac14\sin(2\theta-2\varphi_m)+\frac{1}{2}(\theta-\varphi_m)\right]_0^{2\pi}=\pi.\]
Hence,
\[||\mH_m^{(2)}(\omega)||^2=\left\langle\mH_m^{(2)}(\omega),\mH_m^{(2)}(\omega)\right\rangle=\frac{N}{2}\]
and similarly,
\[||\mH_m^{(3)}(\omega)||^2=\left\langle\mH_m^{(3)}(\omega),\mH_m^{(3)}(\omega)\right\rangle=\frac{N}{2}.\]

\bibliographystyle{elsarticle-num-names}
\bibliography{../../../References}

\end{document}